\documentclass[10pt]{amsart}
\usepackage{palatino} 
\usepackage{pgfplots}
\usepackage{amssymb,amsfonts, amsmath,url,graphicx, latexsym}
\usepackage[all]{xy}
\usepackage{enumitem}
\usepackage{graphicx}
\usepackage{fullpage}
\usepackage{ bbold }

\usepackage{tikz}

\usepackage{tikz-cd}

\usepackage[colorlinks,backref=page]{hyperref}

\newif\ifinfootnote
\infootnotefalse

\let\footnoteasusual\footnote
\renewcommand{\footnote}[1]
{\infootnotetrue\footnoteasusual{#1}\infootnotefalse}

\newcommand{\printname}[1]
{\ifmmode{ \smash{ \raisebox{5pt}{\text{\tiny{#1}}} } }
 \else   {\ifinfootnote \smash{\raisebox{0pt}{\tiny{#1}}}
             \else { \marginpar{
                     \smash{ \makebox[0pt]{\raisebox{-12pt}{\tiny{#1}}} }
                               } } \fi} \fi}

\swapnumbers
\numberwithin{equation}{section}
\newtheorem {Theorem}[equation]         {Theorem}

\newtheorem {Lemma}[equation]           {Lemma}

\newtheorem {Claim*}                    {Claim}

\newtheorem {Corollary} [equation]      {Corollary}
\newtheorem {Proposition}  [equation]   {Proposition}

\theoremstyle{definition}
\newtheorem{Definition}[equation]{Definition}

\theoremstyle{remark}
\newtheorem{Remark}[equation]{Remark}
\newtheorem*{Remark*}{Remark}
\newtheorem{Example}[equation]{Example}

\setlist{topsep=0pt,itemsep=6pt}

\def \Z {{\mathbb Z}}
\def \R {{\mathbb R}}
\def \C {{\mathbb C}}

\newcommand{\Q}{{\mathbb Q}}
\renewcommand{\P}{\textup{P}}

\newcommand{\Aa}{\mathcal{A}}
\newcommand{\K}{{\mathbb K}}

\newcommand{\PP}{\mathcal{P}}

\newcommand{\Sp}{\textup{Sp}}

\newcommand{\B}{\mathbb{B}}
\newcommand{\bb}{\mathbb{b}}

\newcommand{\w}{{\sf w}}
\renewcommand{\v}{{\sf v}}
\newcommand{\fv}{\mathfrak{v}}

\newcommand{\ptconv}{\textup{p-conv}}

\renewcommand{\Sp}{\textup{Sp}} 
\newcommand{\Spec}{\textup{Spec}}
\newcommand{\Proj}{\textup{Proj}}

\renewcommand{\fv}{\mathfrak{v}}

\newcommand{\HH}{\mathcal{H}}
\newcommand{\NN}{\mathcal{N}}
\newcommand{\MM}{\mathcal{M}}

\newcommand{\exampleqed}{%
{   
\renewcommand{\qedsymbol}{$\lozenge$}%
\qed%
}
}

\title{Gorenstein-Fano polytopes and compactifications of rank-2 polyptych lattices }

\author{Adrian Cook}
\address{School of Mathematics, The University of Edinburgh, 
James Clerk Maxwell Building, 
Edinburgh, EH9 3FD , United Kingdom}
\email{a.cook@ed.ac.uk}

\author{Laura Escobar}
\address{Mathematics Department, University of California, Santa Cruz,
1156 High Street
Santa Cruz, CA 95064, USA} 
\email{lauraescobar@ucsc.edu}

\author{Megumi Harada}
\address{Dept.\ of Mathematics and Statistics, McMaster University, 
1280 Main Street West, 
Hamilton, Ontario L8S 4K1, Canada}
\email{haradam@mcmaster.ca}

\author{Christopher Manon} 
\address{Department of Mathematics, 719 Patterson Office Tower, University of Kentucky, Lexington, Kentucky, 40506-0027, USA} 
\email{chris.manon@gmail.com}

\date{\today}


\begin{document}

\begin{abstract}
The theory of polyptych lattices was 
 introduced by Escobar, Harada, and Manon, in an effort to simultaneously generalize combinatorial and geometric phenomena seen in toric, cluster, and tropical geometry, among others. In this companion paper, we give an explicit family of examples of polyptych lattices and give concrete computations to illustrate the abstract theory developed in \cite{EscobarHaradaManon-PL}. Specifically,  we first construct a family of rank-$2$ polyptych lattices $\MM_s$ with $2$ charts, compute their space of points, and prove that they are full and self-dual. We then give a concrete sample computation of a point-convex hull in $\MM_s \otimes \R$ to illustrate that convex geometry in the polyptych lattice setting can exhibit phenomena not seen in the classical situation. We also give multiple examples of $2$-dimensional chart-Gorenstein-Fano polytopes, which give rise to pairs of mutation-related $2$-dimensional (classical) Gorenstein-Fano polytopes. We produce detropicalizations $(\Aa_s, \fv_s)$ of $\MM_s$, and in the case $s=1$ where the detropicalization is a UFD, and with respect to a certain choice of PL polytope $\PP$, we give an explicit generators-and-relations presentation of the (finitely generated) Cox ring of the compactification $X_{\Aa_s}(\PP)$ of $\Spec(\Aa_s)$ with respect to $\PP$. Finally, we illustrate the connection with the theory of complexity-$1$ $T$-varieties using the above Cox ring example.  
\end{abstract}

\maketitle

{
  \hypersetup{linkcolor=black}
  \tableofcontents
}

\section{Introduction}

The theory of \textbf{polyptych lattices} was initiated in \cite{EscobarHaradaManon-PL}, motivated by a desire to simultaneously generalize combinatorial and geometric phenomena observed in toric geometry, cluster varieties, and tropical geometry, among others. This theory builds the notion of mutations directly into the definitions, and gives algebraic-geometric interpretations to these mutations. 
We view the present manuscript as a companion to \cite{EscobarHaradaManon-PL}, and our main purpose is to give an explicit family of examples of polyptych lattices, and give concrete computations. 

Specifically, for a positive integer $s$, the family of polyptych lattices $\MM_s$ studied in this manuscript consists of exactly two charts $\{M_1, M_2\}$, each of rank $2$ (so $M_i \cong \Z^2$), related by a ``shear'' mutation whose length depends on the parameter $s$. Despite its simplicity, this family already exhibits interesting phenomena, as we will see below. Moreover, Oda showed in \cite[Theorem 3]{Oda2025} that the compactifications associated to the rank-$2$ polyptych lattices studied in this manuscript yield all index-one normal projective $\Q$-factorial log Calabi–Yau surface pairs having reduced boundary admitting an effective ample class and a nontrivial torus action.

We now describe the contents of the paper in more detail. After a brief review of the key definitions of \cite{EscobarHaradaManon-PL} in Section~\ref{sec: background}, we define $\MM_s$ in Section~\ref{sec: MM_s}. We state our main results in a brief form and refer the reader to the main text for details. 

\noindent \textbf{Theorem A.} (Proposition~\ref{proposition: points of Sp MMs} and Proposition~\ref{proposition: Ms self dual}) The polyptych lattice $\MM_s$ is full, there exists a concrete description of its space of points $\Sp(\MM_s)$, and $\MM_s$ is self-dual in the sense of \cite[Definition 4.1]{EscobarHaradaManon-PL}.

\medskip

We then explore some convex geometry.  In Section~\ref{sec: pt conv hull example}, we give a sample computation of a \textbf{point-convex hull} of a finite set $S$ in $\MM$ in the sense of \cite[Definition 3.22]{EscobarHaradaManon-PL}. This computation shows that convex geometry can be surprising in the PL context; indeed, our example shows that when viewed in one of the lattices $M_i$, a point-convex hull of a set $S$ may \emph{not} be the same as the classical convex hull. 

In Section~\ref{sec: GF polytope} we give several examples of \textbf{chart-Gorenstein-Fano PL polytopes}, in the sense of \cite[Definitions 5.1, 5.21]{EscobarHaradaManon-PL}. Since our polyptych lattices have two charts, each such PL polytope $\PP$ has two chart images $P_1, P_2$ related by a single mutation. When $\PP$ is chart-Gorenstein-Fano (cf.\ Definition~\ref{definition: PL Gorenstein Fano}), each $P_i$ is a classical $2$-dimensional Gorenstein-Fano polytope, and the two are related by a piecewise-linear map.

The link between polyptych lattices and algebraic geometry comes from our notion of a \textbf{detropicalization} $\Aa_\MM$ of a polyptych lattice $\MM$ \cite[Definition 6.3]{EscobarHaradaManon-PL} and its associated compactification $X_{\Aa_\MM}(\PP)$ with respect to a PL polytope $\PP$ \cite[Section 7.2]{EscobarHaradaManon-PL}. In \cite[Section 7]{EscobarHaradaManon-PL} we established some first geometric properties of these compactifications; in particular, \cite[Theorem 7.19]{EscobarHaradaManon-PL} shows that if $\Aa_\MM$ is a UFD, then $X_{\Aa_\MM}(\PP)$ has finitely generated Cox ring. 
In Section~\ref{sec: detrop MMs} we prove the following; for details, see Theorem~\ref{theorem: MMs detrop}. 

\noindent \textbf{Theorem B.} The polyptych lattice $\MM_s$ is detropicalizable, with detropicalization
$$
\mathcal{A}_s = \mathbb{C}[x_1,x_2,y_1,y_2,y_1^{-1},y_2^{-1}]/\langle x_1x_2 - y_1^s -y_2^s, ~y_2 -1\rangle 
$$
and an appropriately defined valuation $\fv_s: \Aa_s \to \P_{\MM_s}$. 

\medskip
(This construction also produces detropicalizations that are \emph{not} UFDs: for $s=2$, one checks directly that $\Aa_{s}$ is not a UFD.) In Section~\ref{sec: cox ring}, we exploit the fact that $\Aa_{s}$ \emph{is} a UFD when $s=1$ to give an explicit generators-and-relations presentation of the Cox ring of $X_{\Aa_{s}}(\PP)$ for a particular $\PP$. 

\noindent \textbf{Theorem C.} 
(For details see Proposition~\ref{prop: Cox ring for rank 2 ex}.) For a PL polytope $\PP$ as defined in Section~\ref{sec: cox ring}, the Cox ring of the compactification $X_{\Aa_s}(\PP)$ is isomorphic to 
$$
R_\PP:= \C[s_1,s_2, s_3, s_4, x, y] / \langle xy-s_1 s_2 - s_3 s_4 \rangle.
$$
We note that 
Section~\ref{sec: cox ring} also contains a result that is not specific to $\MM_s$: in Proposition~\ref{prop: units} we compute the group of units of a detropicalization $\Aa_\MM$ for any finite polyptych lattice $\MM$ over $\Z$.
 
Finally, in Section~\ref{sec:complexity-one} we explore the connection with complexity-$1$ $T$-varieties and deformations of toric varieties, by comparing our Cox ring computation in Theorem C above, with the description arising from the theory of complexity-$1$ $T$-varieties.
Beyond that computation, we anticipate that a compactification $X_{\Aa_{s}}(\PP)$ arising from a case considered in this paper is a simultaneous deformation of the toric varieties associated to the two chart images $P_1, P_2$ of $\PP$. Such deformations have been studied by Petracci \cite{Petracci2021} and Ilten \cite{Ilten2012, Ilten2011}: applying the mutation between the charts of $\MM$ to $\PP$ should give an instance of a polytope mutation in the sense of Ilten \cite{Ilten2012}, with $X_{\Aa_{s}}(\PP)$ arising as the general fiber of the associated total space. Spaces carrying several toric degenerations related by piecewise-linear mutation also appear in work of Christophersen and Ilten \cite{ChristophersenIlten2016}, and equations cutting out $X_{\Aa_{s}}(\PP)$ could be extracted from work of Petracci \cite{Petracci2021}.

\subsection*{Acknowledgements}

Some of the results contained in this note, particularly those in Sections~\ref{sec: MM_s} and~\ref{sec: detrop MMs}, were obtained in the Master's thesis of the first author, which was supervised by the third author. AC was additionally supported by an NSERC OGS scholarship. LE was supported by NSF CAREER grant DMS-2142656 and DMS-2521270, and a Fields Institute Research Fellowship. MH was supported by a Canada Research Chair Award (Tier 2) and NSERC Discovery Grant 2019-06567. CM is supported by NSF DMS grant 2101911. 
This material is based upon work supported by the National Science Foundation under Grant No. DMS-1928930, while the authors were in residence at the Simons Laufer Mathematical Sciences Institute in Berkeley, California, during the summer of 2025.

\section{Background}\label{sec: background}


In this section we briefly recount some of the basic definitions. For details we refer to \cite{EscobarHaradaManon-PL}. 

We begin with the definition of polyptych lattices.  Recall that a polyptych lattice is a generalization of the concept of lattices; a \textbf{lattice} is a free $\Z$-module of finite rank, and we often fix an identification of a lattice of rank $r$ with $\Z^r$. In this note, we restrict to polyptych lattices over $\Z$ (in the sense of \cite[Definition 2.1]{EscobarHaradaManon-PL}) so we drop the reference to coefficients.

\begin{Definition}\label{definition: polyptych lattice}
    Let $r$ be a positive integer. 
A \textbf{polyptych lattice of rank $r$ (over $\Z$)} is a pair 
\[\mathcal{M} := (\{M_\alpha\}_{\alpha \in \mathcal{I}}, \{\mu_{\alpha,\beta}: M_\alpha \to M_\beta\}_{\alpha,\beta \in \mathcal{I}})\]
consisting of a collection $\{M_\alpha\}_{\alpha \in \mathcal{I}}$ of free $\Z$-modules, each of rank $r$ and indexed by a set $\mathcal{I}$, and a collection of piecewise-linear maps $\mu_{\alpha,\beta}: M_\alpha \to M_\beta$ for every pair $(\alpha,\beta)$ of indices, satisfying the following conditions:
\begin{enumerate} 
\item $\mu_{\alpha,\alpha} = \mathrm{Id}_{M_\alpha}$ is the identity map for all $\alpha \in \mathcal{I}$, 
\item $\mu_{\alpha,\beta} = \mu_{\beta,\alpha}^{-1}$ for all pairs $\alpha,\beta \in \mathcal{I}$, and 
\item $\mu_{\beta,\gamma} \circ \mu_{\alpha,\beta} = \mu_{\alpha,\gamma}$ for all triples $\alpha,\beta,\gamma \in \mathcal{I}$. 
\end{enumerate} 
Note that the requirement (2) above implies that all the maps $\mu_{\alpha,\beta}$ are invertible. We call the maps $\mu_{\alpha,\beta}$ \textbf{mutations}, and we call  $M_\alpha$ a \textbf{chart} of $\mathcal{M}$. When $\mathcal{I}$ is finite, we say $\mathcal{M}$ is a \textbf{finite} polyptych lattice. 
\exampleqed
\end{Definition}

    In this paper, we focus on a class of examples in which $\lvert \mathcal{I} \rvert = 2$, so there are only $2$ charts, and the rank is $2$. In particular, all of the polyptych lattices appearing in this note are finite.

Given a polyptych lattice $\MM$, by slight abuse of notation we denote also by $\MM$ the quotient space 
\begin{equation}\label{eq: def elements of MM} 
\MM := \bigsqcup_{\alpha \in \mathcal{I}} M_\alpha \bigg/ \sim
\end{equation} 
where the equivalence relation is defined by $m_\alpha \sim m_\beta$, for $m_\alpha \in M_\alpha, m_\beta \in M_\beta$, precisely when $\mu_{\alpha,\beta}(m_\alpha)=m_\beta$.  An \textbf{element} of $\MM$ is an equivalence class in the quotient space in~\eqref{eq: def elements of MM}. The \textbf{$\alpha$-th chart map}  
is $\pi_\alpha: \MM \to M_\alpha, \quad m \mapsto m_\alpha$ and we call $\pi_\alpha(m)$ the \textbf{$\alpha$-th coordinate} of $m \in \MM$. 

    Unlike the situation of a classical lattice, there does not exist in general a well-defined operation of addition in $\MM$. Nevertheless, for $m,m' \in \MM$, and $\alpha \in \pi(\MM) = \mathcal{I}$, we may define 
    \begin{equation}\label{eq: addition in chart}
m +_\alpha m' := \pi_\alpha^{-1}(\pi_\alpha(m) + \pi_\alpha(m'))
    \end{equation}
   which we think of as ``addition in the chart $M_\alpha$''. 
  Using this, we can define ``points'' of $\MM$, as below. 

\begin{Definition}
 Let $\MM$ be a polyptych lattice. 
 A \textbf{point of $\MM$} is a  function $p: \MM \to \Z$ such that 
\begin{equation}\label{eq: def point min} 
p(m) + p(m') = \min\{p(m +_\alpha m') \,\mid\, \alpha \in \pi(\MM) \} \, \, \textup{ for all } \, m, m' \in \MM
\end{equation} 
The set of all such $p: \MM \to \Z$ is called \textbf{the space of points of $\MM$} and denoted $\Sp(\MM)$. 
\end{Definition}

Any point $p \in \Sp(\MM)$ induces a function $p_\alpha := p \circ \pi_\alpha^{-1}: M_\alpha \to \Z$ on the lattice $M_\alpha$; these are not linear in general. 

\begin{Definition}\label{definition: full PL}
We let $\Sp(\MM,\alpha)$ denote the subset of points $p$ on $\MM$ such that $p_\alpha: M_\alpha \to \Z$ is linear. If $\Sp(\MM) = \cup_\alpha \Sp(\MM,\alpha)$, then we say that $\MM$ is \textbf{full}. 
\end{Definition}

We need some polyptych lattice analogues of some classical convex-geometric objects. Given a rank $r$ polyptych lattice $\MM$ we may define $\MM_\R$ by tensoring each $M_\alpha$ with $\R$ and using the induced mutation maps.

\begin{Definition}\label{definition: PL cone} 
Let $\MM$ be a polyptych lattice. 
A \textbf{PL cone} is a subset $\mathcal{C} $ of $\MM_{\R}$ such that $\pi_\alpha(\mathcal{C}) \subseteq M_\alpha \otimes\R$ is a rational polyhedral cone for each $\alpha \in \pi(\MM)$ (cf.\ \cite[Definition 1.2.1, Definition 1.2.14]{Cox_Little_Schenck}).  
\end{Definition}

The \textbf{dimension} of a PL cone $\mathcal{C}$ is the dimension of any chart image $\pi_\alpha(\mathcal{C})$. Given a PL cone $\mathcal{C}$, a \textbf{face} $\mathcal{C}'$ of $\mathcal{C}$ is a subset of $\mathcal{C}$ such that $\pi_\alpha(\mathcal{C}')$ is a face of $\pi_\alpha(\mathcal{C})$ for all $\alpha \in \pi(\MM)$. A \textbf{facet} of a PL cone $\mathcal{C}$ is a face of dimension $\dim(\mathcal{C})-1$. Any face of $\mathcal{C}$ is itself a PL cone. 

\begin{Definition}\label{definition: PL fan}
Let $\MM$ be a polyptych lattice. A \textbf{PL fan in $\MM_{\R}$} is a finite collection $\Sigma $ of PL cones in $\MM_{\R}$ such that: 
\begin{enumerate} 
\item for every $\mathcal{C} \in \Sigma$ and every $\alpha \in \pi(\MM)$, the chart image $\pi_\alpha(\mathcal{C})$ is a rational polyhedral cone, 
\item for every $\mathcal{C} \in \Sigma$, each face of $\mathcal{C}$ is also in $\Sigma$, 
\item for all $\mathcal{C},\mathcal{C}' \in \Sigma$, the intersection $\mathcal{C} \cap \mathcal{C}'$ is a face of each, (and hence also in $\Sigma$). 
\end{enumerate}
The \textbf{support} of a PL fan is $\lvert \Sigma \rvert := \cup_{\mathcal{C} \in \Sigma} \mathcal{C}$. A PL fan in $\MM_{\R}$ is \textbf{complete} if $\lvert \Sigma \rvert = \MM_{\R}$. A PL fan $\Sigma'$ \textbf{refines} a PL fan $\Sigma$ if every $\mathcal{C}' \in \Sigma'$ is contained in a PL cone of $\Sigma$ and $\lvert \Sigma' \rvert = \lvert \Sigma \rvert$. 
\end{Definition}

We recall the definition of the PL fan $\Sigma(\MM)$ associated to a polyptych lattice. For any pair $(\alpha,\beta) \in \mathcal{I}^2$, there exists a minimal fan $\Sigma(\MM, \alpha,\beta)$ in $M_\alpha \otimes\mathbb{R}$ such that, for each cone $C \in \Sigma(\MM, \alpha, \beta)$, the restriction $\mu_{\alpha,\beta} \vert_C: C \to M_\beta\otimes\mathbb{R}$ is $\mathbb{R}$-linear. Let $\alpha$ be fixed. Let $\Sigma(\MM, \alpha)$ denote the common refinement of all $\Sigma(\MM,\alpha,\beta)$ as $\beta$ ranges over the finite set $\mathcal{I}=\pi(\MM)$. This is a fan in $M_\alpha \otimes \R$ which has the property that for any cone $C$ of $\Sigma(\MM, \alpha)$ and any $\beta \in \mathcal{I}$, the  mutation $\mu_{\alpha,\beta}$ restricts to $C$ to be linear. 
Now let $\MM_\R = \bigcup_{C \in \Sigma(\MM,\alpha)} \pi_\alpha^{-1}(C)$ be the decomposition of $\MM_\R$ into preimages of the cones in $\Sigma(\MM,\alpha)$. It is shown in \cite[Lemma 2.10]{EscobarHaradaManon-PL} that this decomposition is a PL fan. We call the union of these cones the \textbf{PL fan of $\MM$}, and denote it by $\Sigma(\MM)$. 

Given two polyptych lattices $\MM$ and $\NN$, we say that the two are strictly dual to each other if - roughly speaking - we can identify (the elements of) $\MM$ with $\Sp(\NN)$, and vice versa, and their PL fans are compatible. The precise version is below. 

\begin{Definition} 
\label{def_dual}
Let $\MM, \NN$ be polyptych lattices and $\v: \MM \to \Sp(\NN)$ and $\w: \NN \to \Sp(\MM)$ a pair of maps.  We say that $\v, \w$ are a \textbf{strict dual pairing} if: 
\begin{enumerate} 
\item[(1)] $\v(m)(n) = \w(n)(m)$ for all $n \in \NN, m \in \MM$, 
\item[(2)] $\v$ and $\w$ are both bijections, and
\item[(3)] the preimages $\v^{-1}\Sp_\R(\NN,\beta)$ (respectively $\w^{-1}\Sp_\R(\MM,\alpha)$) are precisely the maximal-dimensional faces of $\Sigma(\MM)$ (respectively $\Sigma(\NN)$). 
\end{enumerate} 
If $\MM$ has a  (strict) dual pairing with itself with respect to a single map $\mathbf{w}: \MM \to \Sp(\MM)$, we say $\MM$ is \textbf{(strictly) self-dual. }
\end{Definition}

It is shown in \cite[Lemma 3.5]{EscobarHaradaManon-PL} that for any finite polyptych lattice $\NN$, any point $p \in \Sp(\NN)$ extends naturally to a piecewise linear function, also denoted $p$, on $\NN_\R$; see \cite{EscobarHaradaManon-PL} for precise definitions. Let $\P_\NN$ denote the set of piecewise linear functions on $\NN_\R$ generated by $\Sp(\NN)$ under the operations $+$ and $\mathrm{min}$; then this set $\P_\NN$ is an idempotent $\Z_{\geq 0}$-semialgebra with respect to these operations. We refer to $\P_\NN$ as the \textbf{point semialgebra of $\NN$}. We may equip $\P_\NN$ with the partial order defined by $a \geq b$ if and only if $\min\{a,b\}=b$, where the $\min$ is the pointwise minimum of functions.

For the purposes of this paper, we need only define valuations with values in either $\P_\NN$ or $\Z$, so we restrict to these cases. 
Let $\K$ be an algebraically closed field and $\Aa$ be a Noetherian $\K$-algebra which is an integral domain. 
A valuation $\fv: \Aa \to \P_\NN$ (resp. $\fv: \Aa \to \Z$) is an analogue of a classical discrete valuation on a field.  We have the following. 

\begin{Definition}\label{def_valuation}
    We say a map $\fv: \Aa \to \P_\NN\cup\{\infty\}$ (resp. $\fv: \Aa \to \Z\cup\{\infty\}$) is a \textbf{valuation with values in $\P_\NN$} (resp. $\Z$) if for all $f,g\in\Aa$ we have:
\begin{enumerate}
    \item $\fv(fg)=\fv(f)\odot\fv(g)$, 
    \item $\fv(f+g)\ge \fv(f)\oplus\fv(g)$, 
    \item $\fv(cf)=\fv(f)$, for all $c\in \K^*$, and 
    \item $\fv(0)=\infty$. 
\end{enumerate}
\end{Definition}

We may now define detropicalizations of polyptych lattices in the case when $\MM$ has a strict dual. We restrict to this case since the examples in this note have strict duals. The definition in \cite{EscobarHaradaManon-PL} uses valuations valued in the canonical semialgebra $S_\MM$ of $\MM$ (which we have not defined here), but it is shown in \cite[Proposition 4.9]{EscobarHaradaManon-PL} that $S_\MM \cong \P_\NN$ when $\MM$ and $\NN$ are strict duals, so here we may take the codomain to be $\P_\NN$. 

\begin{Definition}\label{definition: detropicalization}
Let $\MM$ be a finite polyptych lattice. Assume that $\MM$ has a strict dual $\NN$. Let $\Aa_\MM$ be a Noetherian $\K$-algebra which is an integral domain. Let $\fv: \Aa_\MM \to \P_\NN$ be a valuation with values in $\P_\NN$. We say that the pair $(\Aa_\MM, \fv)$ is a \textbf{detropicalization of $\MM$} if every element of $ \Sp(\NN)\subset \P_\NN$ is in the image of $\fv$, and the Krull dimension of $\Aa_\MM$ equals the rank $r$ of $\MM$.
We say that a $\K$-vector space basis $\B$ of $\Aa_\MM$ is a \textbf{convex adapted basis} for $\fv: \Aa_\MM \to \P_\NN$ if $\fv(\sum \lambda_i \bb_i) = \bigoplus_i \fv(\bb_i) = \min_i \{\fv(\bb_i)\}$, for any finite collection $\lambda_i \in \K^*$ and $\bb_i \in \B$, and
$\fv(\bb) \in \Sp(\NN) \subset \P_\NN$ for all $\bb \in \B$.
\end{Definition}

Let $p \in \Sp(\MM)$ be a point of $\MM$. Let $a \in \Z$. The \textbf{PL half-space} with threshold $a$ associated to $p$ is
\begin{equation}\label{eq: def M half space} 
\HH_{p, a} := \{ m \in \MM_\R \, \mid \, p(m) \geq a\} \subset \MM_\R. 
\end{equation} 
A set $\PP$ is a \textbf{PL polytope} if it is compact\footnote{The topology on $\MM$ is the quotient topology coming from its definition \eqref{eq: def elements of MM}.} and it is a finite intersection of PL half spaces, i.e., 
$$\PP = \bigcap_{i=1}^\ell \HH_{p_i, a_i}$$
for some collection of points $p_i \in \Sp(\MM)$ and $a_i \in \Z$. The \textbf{set of vertices $V(\PP)$} of $\PP$ is 
\begin{equation}\label{eq: def vertices of P}
V(\PP) := \{ m \in \MM_{\R}  \, \mid \, \exists \alpha \in \pi(\MM), \, \, \pi_\alpha(m)\,  \textup{ is a vertex of } \, \pi_\alpha(\PP)\}. 
\end{equation} 
Vertices need not be elements of $\MM$. 
We say that $\PP$ is an \textbf{integral PL polytope} if $\pi_\alpha(\PP)$ is an integral polytope in $M_\alpha \otimes \R$ (i.e., all its vertices are in $M_\alpha$) for every $\alpha \in \pi(\MM)$.

\begin{Definition}\label{definition: PL Gorenstein Fano} 
Let $\MM$ be a finite polyptych lattice over $\Z$. We say that a PL polytope $\PP$ in $\MM_\R$ is \textbf{chart-Gorenstein-Fano} if $\PP$ is a full-dimensional integral PL polytope, and, its PL half-space representation is of the form 
$$
\PP = \bigcap_{i=1}^{\ell} \HH_{p_i, -1}
$$
where $p_i \in \Sp(\MM)$ and $a_i=-1$ for all $i \in [\ell]$. 
\exampleqed
\end{Definition} 

Later in this manuscript, we give multiple explicit examples of chart-Gorenstein-Fano PL polytopes. 
Moreover, following \cite{EscobarHaradaManon-PL} and in the setting when $\MM$ possesses a strict dual, we also have a theory of dual polytopes. Indeed, in the presence of a strict dual $\NN$ to $\MM$ we define 
the \textbf{support function} $\psi_\PP:\NN_{\R} \to \R$ of $\PP$ as
\begin{equation}\label{eq: def support function}
\psi_\PP(-) := \min\{\v(u)(-) \mid u \in \PP\}.
\end{equation} 
Then the \textbf{dual PL polytope $\PP^\vee$ to $\PP$} (with respect to the strict dual $\NN$) is 
\begin{equation}\label{eq: def dual Pcheck}
\PP^\vee := \{n \in \NN_\R \, \mid \, \psi_\PP(n) \geq -1\} \subset \NN_\R.
\end{equation}
It is shown in \cite[Lemma 5.16]{EscobarHaradaManon-PL} that $\PP^\vee$ can be expressed as 
 \begin{equation}\label{eq: Pvee as N polytope}
  \PP^\vee = \bigcap_{m \in V(\PP)} \HH_{\v(m),-1}
 \end{equation}
 and $\PP^\vee$ is compact in $\NN_\R$. 

 \begin{Remark} 
It should be emphasized here that, in our setting of PL polytopes, it is not necessarily true that the dual of a chart-Gorenstein-Fano polytope is an integral polytope.  See Example~\ref{example: dual not integral}. We intend to explore these subtleties of PL convex geometry in future work. 
 \end{Remark}

\section{The rank-$2$ polyptych lattice $\MM_s$ and its space of points}\label{sec: MM_s}

As mentioned in the introduction, one of the goals of this note is to construct a concrete family of rank-$2$ polyptych lattices which serve to illustrate the abstract theory introduced in \cite{EscobarHaradaManon-PL}. In this section, we will define our family of polyptych lattices, compute the associated spaces of points, and show that they are full and strictly self-dual.

Let $s$ be a positive integer. We define a polyptych lattice $\MM_s$ associated to $s$ as follows.
There are two coordinate charts $M_1$ and $M_2$, both isomorphic to $\Z^2$, so the rank $r$ is $2$ and the set of charts $\mathcal{I} = \{1,2\}$. We fix once and for all identifications of $M_1$ and $M_2$ with $\Z^2$, and use coordinates $(x,y) \in \Z^2$ on $M_1$, and $(u,v) \in \Z^2$ on $M_2$. 
To specify the mutations, it suffices to describe the piecewise-linear mutation map $\mu_{1,2}: M_1 \rightarrow M_2$ as follows: 
\begin{equation}\label{eq: def mutation M_s}
\begin{split}
\mu_{1,2}(x,y) & = (\min\{0, s \, y\} - x, y) \\
& = \begin{cases} 
(-x, y) \quad \quad \textup{ if } y \geq 0 \\
(s \, y-x, y) \quad \textup{ if } y < 0.  
\end{cases}
\end{split}
\end{equation}
We see that $(\mu_{1,2})_{\R}$ is continuous and the domains of linearity of $\mu_{1,2}$ are the upper- and lower-half spaces $\{y \geq 0\}$ and $\{y \leq 0\}$ of $M_1$, where the mutation may be represented, respectively,  by the matrices $\begin{bmatrix} -1 & 0 \\ 0 & 1 \end{bmatrix}$ and $\begin{bmatrix} -1 & s \\ 0 & 1 \end{bmatrix}$. 
It is also straightforward to compute that the inverse mutation $\mu_{2,1}: M_2 \to M_1$ is given by the same formula, 
$$
\mu_{2,1}(u,v) = (\min\{0,s\, v\}-u,v)
$$
and thus also has two domains of linearity, $\{v \geq 0\}$ and $\{v \leq 0\}$. Since there are only 2 charts, we will also refer to $(x,y)$ and $(u,v)$ as the \emph{first and second coordinates} respectively (of an element of $\MM_s$), and $M_1$ as the \emph{first chart}, $M_2$ as the \emph{second chart}.  The maximal cones of the PL fan $\Sigma(\MM_s)$ consist of the two subsets $H_+ := \pi_1^{-1}(\{y \geq 0\}) = \pi_2^{-1}(\{v \geq 0\}) \subset \MM_s$ and $H_- := \pi_1^{-1}(\{y \leq 0\}) = \pi_2^{-1}(\{v \leq 0\}) \subset \MM_s$.

The following lemma is proven in \cite{EscobarHaradaManon-PL}.

\begin{Lemma} 
Let $\MM$ be a finite polyptych lattice and let $p \in \Sp(\MM)$. Let $C$ be a cone in the PL fan $\Sigma(\MM)$ of $\MM$. Then $p$ is linear when restricted to $C$. 
\end{Lemma}

Using the above lemma, we can explicitly compute the space of points $\Sp(\MM_s)$ of $\MM_s$. 
Indeed, by the lemma, we know that for any point $p \in \Sp(\MM_s)$, the induced functions $p_i = p \circ \pi_i^{-1}$ must be linear on the upper-half and lower-half spaces of $M_i$, so both $p_1$ and $p_2$ are completely specified by two linear functions on these two half-spaces. With this in mind, we set the following notation. Let $\{\mathbf{e}_1,\mathbf{e}_2\}$ denote the standard basis for $\Z^2$.  
Consider the following elements of $\MM_s$: ${\mathfrak{e}}_1:=\pi_1^{-1}(\mathbf{e}_1)$, $\mathfrak{e}_2:=\pi_1^{-1}(\mathbf{e}_2)$, and $\mathfrak{e}'_2:=\pi_1^{-1}(-\mathbf{e}_2)$. Note also that, since an element $m \in \MM$ is completely determined by its first coordinate $\pi_1(m)$, any function $p: \MM \to \Z$ is uniquely determined by the induced function $p_1 := p \circ \pi_1^{-1}$. We take advantage of this observation in the proposition below. 
We have the following.

\begin{Proposition}\label{proposition: points of Sp MMs}
Let $p(\mathfrak{e}_1)$, $p(\mathfrak{e}_2)$, and $p(\mathfrak{e}'_2)$ denote integers chosen such that $p(\mathfrak{e}_2)+p(\mathfrak{e}'_2)=\min\{0,s \cdot p(\mathfrak{e}_1)\}$. 
Let $p: \MM_s \to \Z$ be the function uniquely specified by  
\begin{equation}\label{eq_ex_point}
    p_1(x,y) := p \circ \pi_1^{-1}(x,y)=
    \begin{cases}
        x\cdot p(\mathfrak{e}_1)-y\cdot p(\mathfrak{e}'_2), & y\le 0
        \\
        x \cdot p(\mathfrak{e}_1)+y \cdot p(\mathfrak{e}_2), & y\ge 0
    \end{cases}
.\end{equation}
Then $p$ is a point on $\MM_s$, and, any point in $\Sp(\MM_s)$ is of this form. 
In particular, $\MM_s$ is full, and 
 $\Sp(\MM_s) \otimes \mathbb{R}$ can be identified with the subset $\mathcal{T}_s$ of $\mathbb{R}^3$ defined as 
 \begin{equation}\label{eq: def Tau_s}
 \mathcal{T}_s := \{(a,b,c) \in \R^3 \, \mid \, a+b = \min\{0,s\cdot c\}\} \subset \R^3.
 \end{equation}
\end{Proposition}

\begin{proof} 
We first show that a function $p: \MM_s \to \Z$ defined by~\eqref{eq_ex_point} is in $\Sp(\MM_s)$. To prove this, we must check the condition~\eqref{eq: def point min}. We may compute in terms of $p_1$ instead of $p$, where the requirement becomes that for all $(x,y), (x',y') \in M_1$, we have 
\begin{equation}\label{eq: LHS for MM_s}
p_1(x,y) + p_1(x',y') = \min\{p_1(x+x',y+y'), p_1(\min\{0,s(y+y')\} - \min\{0,sy\} - \min\{0,sy'\}+x+x', y+y')\}
\end{equation}
where the second expression in the minimum is equal to $\mu_{2,1}(\mu_{1,2}(x,y)+\mu_{1,2}(x',y'))$. (This is the first coordinate of the addition of $\pi_1^{-1}(x,y)$ and $\pi_1^{-1}(x',y')$ in the chart $M_2$ as in~\eqref{eq: addition in chart}.)

To check~\eqref{eq: LHS for MM_s}, we take cases. Note that $p$ is linear when restricted to $H_+$ or $H_-$ so we only need to check the cases in which the $m$ and $m'$ are contained in distinct cones of linearity. 
Consider first the case when $m \in H_+, m' \in H_-$, and $m +_i m' \in H_+$ for $i=1,2$. The LHS of~\eqref{eq: LHS for MM_s} is then 
\begin{equation}\label{eq: final LHS}
x p(\mathfrak{e}_1) + y p(\mathfrak{e}_2) + x' p(\mathfrak{e}_1) - y' p(\mathfrak{e}'_2).
\end{equation}
The RHS of~\eqref{eq: LHS for MM_s} can be simplified using that $y \geq 0, y' \leq 0, y+y' \geq 0$, and we obtain
$$
\min\{(x+x')p(\mathfrak{e}_1)+(y+y')p(\mathfrak{e}_2), (-sy' + x+x')p(\mathfrak{e}_1) + (y+y')p(\mathfrak{e}_2)\}
$$
which is in turn equal to 
\begin{equation}\label{eq: final RHS}
 (x+x')p(\mathfrak{e}_1) + (y+y')p(\mathfrak{e}_2) + \min\{0, -sy' \, p(\mathfrak{e}_1)\} = 
  (x+x')p(\mathfrak{e}_1) + (y+y')p(\mathfrak{e}_2) -sy' \min\{0,p(\mathfrak{e}_1)\}
\end{equation}
where the last equality follows because $s \geq 0, y' \leq 0$ implies $-sy' \geq 0$. 
Setting~\eqref{eq: final LHS} equal to~\eqref{eq: final RHS} the condition becomes 
$$
-y' \, p(\mathfrak{e}'_2) = y' \, p(\mathfrak{e}_2) - sy' \min\{0, p(\mathfrak{e}_1)\}
$$
where this equality must hold for any $y' \leq 0$. This is true if and only if $p(\mathfrak{e}_2)+p(\mathfrak{e}'_2) = s \cdot \min\{0,p(\mathfrak{e}_1)\}$. 
Checking the other case when $m \in H_+, m' \in H_-$ and $m+_i m' \in H_-$ is similar and is left to the reader. In this case we also obtain that the condition of being a point is satisfied if and only if $p(\mathfrak{e}_2)+p(\mathfrak{e}'_2) = s \cdot \min\{0,p(\mathfrak{e}_1)\}$. Thus we conclude that $p$ is a point in $\MM_s$, and moreover, if $p$ is a point in $\MM_s$, then the values $p(\mathfrak{e}_1), p(\mathfrak{e}_2), p(\mathfrak{e}'_2)$, which correspond to the values of $p$ on the elements $\mathfrak{e}_1,\mathfrak{e}_2,\mathfrak{e}'_2$ respectively, must satisfy $p(\mathfrak{e}_2)+p(\mathfrak{e}'_2) = s \cdot \min\{0,p(\mathfrak{e}_1)\}$. This proves the first statement of the proposition.  

To see that $\MM_s$ is full, it suffices to show that any point $p$ in $\Sp(\MM_s)$ is linear in either the first chart or the second chart. We know that $p(\mathfrak{e}_2)+p(\mathfrak{e}'_2) = s \cdot \min\{0,p(\mathfrak{e}_1)\}$, so let us take cases. Suppose $p(\mathfrak{e}_2)+p(\mathfrak{e}'_2) = 0$. Then $p(\mathfrak{e}_2) = - p(\mathfrak{e}'_2)$. From the definition of $p$ from~\eqref{eq_ex_point} it follows immediately that, in this case, $p_1 = p \circ \pi_1^{-1}$ is linear on all of $M_1$, so $p \in \Sp(\MM,1)$. On the other hand, suppose that $p(\mathfrak{e}_2) + p(\mathfrak{e}'_2) = s  \cdot p(\mathfrak{e}_1) <0$. Then $p(\mathfrak{e}_2)=s \cdot p(\mathfrak{e}_1)-p(\mathfrak{e}'_2)$ so that we may write 
\begin{equation*}
    p((x,y),\mu_{1,2}(x,y))=
    \begin{cases}
        x\cdot p(\mathfrak{e}_1)-y\cdot p(\mathfrak{e}'_2), & y\le 0
        \\
        (x +sy) \cdot p(\mathfrak{e}_1) - y \cdot p(\mathfrak{e}'_2), & y\ge 0
    \end{cases}
.\end{equation*}
Rewriting this in the coordinates for the second chart, we have 
\begin{equation}\label{eq_ex_point M2}
    p(\mu_{2,1}(u,v),(u,v))=
    \begin{cases}
        (sv-u) \cdot p(\mathfrak{e}_1)-v \cdot p(\mathfrak{e}'_2), & v \le 0
        \\
        (sv-u) \cdot p(\mathfrak{e}_1) - v \cdot p(\mathfrak{e}'_2), & v\ge 0
    \end{cases}
\end{equation}
which shows that, in this case, $p_2 := p \circ \pi_2^{-1}$ is linear, i.e. $p$ is linear in the second chart, and $p \in \Sp(\MM_s,2)$. Thus any point in $\Sp(\MM_s)$ is linear in one of the coordinate charts, so $\Sp(\MM_s) = \Sp(\MM_s,1) \cup \Sp(\MM_s,2)$ and $\MM_s$ is full. Finally, it follows from the above that the space of points $\Sp(\MM_s) \otimes \R$ may be identified with the set of parameters $\{(a,b,c): a+b=\min\{0,s \cdot c\}\} \subset \R^3$, given by the choices of the values $p(\mathfrak{e}_1), p(\mathfrak{e}_2), p(\mathfrak{e}'_2)$, so the last claim follows. 
\end{proof}

\begin{Remark} 
In the proposition above, we express a point $p$ in $\Sp(\MM_s)$ as a function of the variables of the first coordinate chart $M_1$. For later computations, it will also be useful to express $p$ in terms of the $M_2$ coordinates. It is straightforward to compute that, given the parameters $p(\mathfrak{e}_2), p(\mathfrak{e}'_2), p(\mathfrak{e}_1)$ as in Proposition \ref{proposition: points of Sp MMs}, $p \in \Sp(\MM_s)$ expressed in $M_2$ coordinates $(u,v)$ is given by 
\begin{equation}\label{eq: point in M2 coord} 
p(\mu_{2,1}(u,v),(u,v)) = 
\begin{cases} 
- u \cdot p(\mathfrak{e}_1) + v \cdot (s \cdot p(\mathfrak{e}_1) - p(\mathfrak{e}'_2)), \, \quad \textup{ if }  \, v \leq 0 \\ 
- u \cdot p(\mathfrak{e}_1) + v \cdot p(\mathfrak{e}_2), \quad \quad \quad \quad \quad \textup{ if } \, v \geq 0.
\end{cases} 
\end{equation}
\exampleqed
\end{Remark}

We next claim that $\MM_s$ is self-dual in the sense of Definition \ref{def_dual}, i.e., there exists a strict dual pairing of $\MM_s$ with itself. This means that we seek a bijective mapping 
$\mathbf{w}_s: \MM_s \to \Sp(\MM_s)$ such that for any $m, m' \in \MM_s$, we have 
$$
\mathbf{w}_s(m)(m') = \mathbf{w}_s(m')(m)
$$
and such that preimages of $\Sp(\MM_s,i)$ for $i=1,2$ land precisely on the maximal-dimensional faces of $\Sigma(\MM_s)$. 
Recall that by Proposition~\ref{proposition: points of Sp MMs} a point $p$ in $\Sp(\MM_s)$ is completely determined by a triple $(p(\mathfrak{e}_2), p(\mathfrak{e}'_2),p(\mathfrak{e}_1))$ of integers in $\mathcal{T}_s$. More precisely we have a bijection 
\begin{equation}\label{eq: Ts identified with Sp MMs}
\psi: \Sp(\MM_s) \to \mathcal{T}_s \cap \Z^3, \quad \quad 
p \mapsto (p(\mathfrak{e}_2), p(\mathfrak{e}'_2),p(\mathfrak{e}_1)).
\end{equation}
 For the remainder of this discussion we identify $\Sp(\MM_s)$ with $\mathcal{T}_s$ via $\psi$ and as such, we will define below a function $\mathbf{w}_s: \MM_s \to \mathcal{T}_s$ and interpret this as a mapping to $\Sp(\MM_s)$. 

Let $m=\pi_1^{-1}(x,y)$. 
We then define
\begin{equation}\label{eq: def ws}
\mathbf{w}_s(m)=\mathbf{w}_s(\pi_1^{-1}(x,y)) = 
\begin{cases} 
(x, -x, y) \quad \quad \textup{ if } \, \, y \geq 0 \\
(x, sy-x, y) \qquad \textup{ if } \, \, y \leq 0. 
\end{cases} 
\end{equation}

\begin{Proposition}\label{proposition: Ms self dual}
The map $\mathbf{w}_s$ of~\eqref{eq: def ws} defines a strict self-dual pairing of $\MM_s$ with itself. 
\end{Proposition}

\begin{proof} 
We must check the conditions (1),(2),(3) of Definition~\ref{def_dual} for $\MM=\NN=\MM_s$ and $\mathbf{v}=\mathbf{w}$, where $\mathbf{w}$ is defined in~\eqref{eq: def ws}. 

We first prove (1). We take cases. First suppose $m,m' \in H_+$.  Then $m=\pi_1^{-1}(x,y), m'=\pi_1^{-1}(x',y')$, where $y,y' \geq 0$. To check that $\mathbf{w}_s(m)(m') = \mathbf{w}_s(m')(m)$, we compute both sides. The LHS is 
$$\mathbf{w}_s(\pi_1^{-1}(x,y))(\pi_1^{-1}(x',y')) = y \cdot x' + x \cdot y'$$
because $\mathbf{w}_s(\pi_1^{-1}(x,y))$ is defined to be $(x,-x,y)$, i.e. $p(\mathfrak{e}_2)=x, p(\mathfrak{e}'_2)=-x, p(\mathfrak{e}_1)=y$, so the computation follows from~\eqref{eq_ex_point}. The RHS may similarly computed to be $y' \cdot x + x' \cdot y$, and hence the equality holds. 

Next suppose that $m,m' \in H_-$. In this case we have $m=\pi_1^{-1}(x,y), m'=\pi_1^{-1}(x',y')$ with $y \leq 0$, and by definition $\mathbf{w}_s(\pi_1^{-1}(x,y)) = (x,sy-x, y) \in \mathcal{T}_s$. It follows from~\eqref{eq_ex_point} that we have 
$$
\mathbf{w}_s(\pi_1^{-1}(x,y))(\pi_1^{-1}(x',y')) = yx' - syy' + xy'
$$
The RHS may be computed similarly to be 
$$
\mathbf{w}_s((x',y'),(sy'-x',y'))((x,y),(sy-x,y)) = y'x - sy' \, y + x'y
$$
so the two sides are equal, as desired. 
Finally, for the case $m \in H_+, m' \in H_-$, similar computations show that the LHS is equal to $yx' + xy'$ and the RHS is equal to $y'x + x'y$, so they are again equal. By symmetry, the equality holds also for the case $m \in H_-, m' \in H_+$. 
 This concludes the proof of (1). 

The condition (2) of Definition~\ref{def_dual} follows immediately since the map is evidently injective, as the three parameters $p(\mathfrak{e}_1), p(\mathfrak{e}_2), p(\mathfrak{e}'_2)$ completely determine $p$, and is surjective by the claim of Proposition~\ref{proposition: points of Sp MMs}. 

It remains to prove the condition (3). We have seen in the discussion above that $\Sigma(\MM_s)$ consists of the two cones of linearity $H_+$ and $H_-$. Moreover, in the proof of Proposition \ref{proposition: points of Sp MMs} we saw that $\Sp(\MM_s,1) = \{ p \, \mid \, p(\mathfrak{e}_2)+p(\mathfrak{e}'_2) = 0 = \min\{0,s \cdot c\}\}$. In other words, in terms of coordinates on $\mathcal{T}_s$, the subset $\Sp(\MM_s,1)$ corresponds to $\{c \geq 0\} = \{a+b = \min\{0, s \cdot c\}=0\}$. Now from~\eqref{eq: def ws} it follows that the preimage under $\mathbf{w}_s$ of the subset $\{c \geq 0\}$ precisely $\{y \geq 0\} = H_+$. By a similar argument, $\Sp(\MM_s,2)$ is identified with $\{c \leq 0\} = \{a+b = \min\{0,s \cdot c\}=s \cdot c\}$, which again from~\eqref{eq: def ws} can be seen to have preimage $H_-$. Thus the preimages of $\Sp(\MM_s,i)$ for $i=1,2$ correspond precisely to the maximal-dimensional cones of $\Sigma(\MM_s)$, as desired. This concludes the proof. 
\end{proof}

\section{Example: a point-convex hull in $\MM_s \otimes \R$}\label{sec: pt conv hull example}

In this section, we illustrate via one sample computation that convex geometry in the context of polyptych lattices can exhibit phenomena that are not intuitive from the classical perspective. First we recall some definitions from \cite{EscobarHaradaManon-PL}. Given a subset $S \subset \MM_\R$ we define the point-convex hull of $S$, denoted $\ptconv_\R(S)$, to be
\begin{equation}\label{eq: def pt convex hull}
\ptconv_\R(S) := \bigcap_{S \subset \HH_{p,\lambda}} \HH_{p,\lambda}
\end{equation}
where $p \in \Sp(\MM), \lambda \in \R$, and the intersection ranges over those choices $p,\lambda$ with $S \subset \HH_{p,\lambda}$. Point-convexity is a natural polyptych-lattice analogue of the classical notion of convexity. 

For this discussion we fix $s=1$, so the mutation $\mu_{1,2}: M_1 \to M_2$ is given by $\mu_{1,2}(x,y) = (\min\{0,y\} - x, y)$. Now fix the (finite) set $S:= \{\pi_1^{-1}(0,0),\pi_1^{-1}(0,1),\pi_1^{-1}(0,-1)\} \subset \MM_s$. Then 
$$\pi_2(S) = \{(-1,-1), (0,0),(0,1)\} \subset M_2.$$ We illustrate $S$ in each of the coordinate charts in Figure~\ref{fig_set_S} below.

\begin{figure}[h]
\centering
\begin{tikzpicture}
     \draw[gray,<->] (-1.5,0)--(1.5,0);
     \draw[gray,<->] (0,-1)--(0,1);
    \foreach \i in {-2,...,2}
      \foreach \j in {-2,...,2}{
        \filldraw[black] (.5*\i,.5*\j) circle(.5pt);
      };
    \filldraw[black, thick] (.5*0,.5*1) circle(1pt);
    \filldraw[black, thick] (.5*0,.5*0) circle(1pt);
    \filldraw[black, thick] (.5*0,.5*-1) circle(1pt);
\end{tikzpicture}
 \hspace{5cm}
 \begin{tikzpicture}
     \draw[gray,<->] (-1.5,0)--(1.5,0);
     \draw[gray,<->] (0,-1)--(0,1);
    \foreach \i in {-2,...,2}
      \foreach \j in {-2,...,2}{
        \filldraw[black] (.5*\i,.5*\j) circle(.5pt);
      };
    \filldraw[black, thick] (.5*0,.5*1) circle(1pt);
    \filldraw[black, thick] (.5*0,.5*0) circle(1pt);
    \filldraw[black, thick] (.5*-1,.5*-1) circle(1pt);
\end{tikzpicture}
 \caption{The two chart images of the set $S$. On the left is $\pi_1(S)$ and on the right is $\pi_2(S)$. In what follows, we compute the point-convex hull of $S$. }
    \label{fig_set_S}
\end{figure}
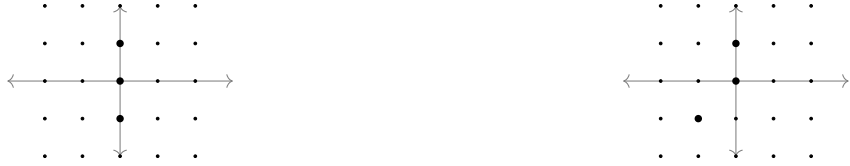

We now compute $\ptconv_\R(S)$ and we will see that $\pi_i(\ptconv_\R(S))$ is not the same as the classical convex hull of $\pi_i(S)$ in $M_i \otimes \R$.

To compute $\ptconv_\R(S)$, by its definition~\eqref{eq: def pt convex hull}, we must first identify those PL half-spaces $\HH_{p,\lambda}$ with the property that $S \subset \HH_{p,\lambda}$, and then we must take the intersection of all of them. We have already seen above that any $p \in \Sp(\MM_s)$ is of the form 
$$
p_1(x,y) = p \circ \pi_1^{-1}(x,y)  = 
\begin{cases} 
cx-by, \quad \textup{ if } y \leq 0 \\
cx+ay, \quad \textup{ if } y \geq 0 
\end{cases} 
$$
for a triple $(a,b,c) \in \R^3$ satisfying $a+b = \min\{0,c\}$. To analyze the behavior of various pairs of $p$ and $\lambda$, we take cases. 

First suppose $c=0$. Then $a+b=0$ and $p_1(x,y)=-by$ for all $(x,y) \in M_1$. Now if $S \subset \HH_{p,\lambda}$ then $p_1(0,1)=-b \geq \lambda, p_1(0,0) = 0 \geq \lambda$, and $p_1(0,-1)=b \geq \lambda$. So $\lambda \leq 0$ and $b \in [\lambda, -\lambda]$. If $b=0$ then $p_1 \equiv 0$ and $\pi_1(\HH_{p,\lambda}) = M_1 \otimes \R$ so this case is trivial and we may omit it from consideration. For $b \neq 0$, it can be seen that for such a $p$ and $\lambda$, we can describe $\pi_1(\HH_{p,\lambda})$ as follows. If $b<0$ then 
$$
\pi_1(\HH_{p,\lambda}) = \{(x,y) \in M_1 \otimes \R \, \mid \, y \geq - \frac{\lambda}{b} \}
$$
and if $b>0$ then 
$$
\pi_1(\HH_{p,\lambda}) = \left\{(x,y) \in M_1 \otimes \R \, \mid \, y \leq -\frac{\lambda}{b} \right\}. 
$$
Note that from the condition $b \in [\lambda,-\lambda]$ it follows that $\lvert \frac{\lambda}{b}\rvert \geq 1$. 
Second, suppose $c>0$. A similar computation shows that for $S \subset \HH_{p,\lambda}$ to hold we again must have $\lambda \leq 0$ and $b \in [\lambda,-\lambda]$ and thus, for such $p$ and $\lambda$, we have 
\begin{equation}
    \begin{split}
        \pi_1(\HH_{p,\lambda}) & = \{(x,y) \in M_1 \otimes \R \, \mid \, cx-by \geq \lambda\} \\
        & = \{(x,y) \in M_1 \otimes \R \, \mid \, x-b'y \geq \lambda' \} \textup{ where } \, \lambda' = \frac{\lambda}{c} \leq 0, b' = \frac{b}{c} \in [\lambda', -\lambda'].
    \end{split}
\end{equation}
We depict the different possibilities when $c>0$ in Figure~\ref{fig: c>0 cases}.

\begin{figure}[h]
\centering
\begin{tikzpicture} \begin{scope}[scale=1.6]
     \filldraw[fill=gray!50,draw=none] 
     (-.5*1.7,-.5*2)--(-.5*1.7,.5*2)--(.5*2,.5*2)--(.5*2,-.5*2)
    ;
    
     \draw[gray,<->] (-1.15,0)--(1,0);
     \draw[gray,<->] (0,-1)--(0,1);
     \draw[black] (-.5*1.7,.5*2)--(-.5*1.7,-.5*2);
    \foreach \i in {-2,...,2}
      \foreach \j in {-2,...,2}{
        \filldraw[black] (.5*\i,.5*\j) circle(.5pt);
      };
    \draw (-.5*.3,0) node [below] {\tiny $x\geq \lambda', \lambda'\leq 0$};
    \draw (0, -.5*2) node [below] {\footnotesize Case $b'=0$};
    \end{scope}
\end{tikzpicture}
 \hspace{2cm}
 \begin{tikzpicture}\begin{scope}[scale=1.6]
     \filldraw[fill=gray!50,draw=none] (-.5*3,-.5*.5)--(-.5*2.25,0)--(0,.5*1.5)--(.5*0.75,.5*2)--(.5*2,.5*2)--(.5*2,-.5*2)--(-.5*3,-.5*2)--(-.5*3,-.5*.5)
    ;

    \draw (-.5*3,-.5*.5)--(-.5*2.25,0)--(0,.5*1.5)--(.5*0.75,.5*2)
    ;
     \draw[gray,<->] (-1.5,0)--(1,0);
     \draw[gray,<->] (0,-1)--(0,1);

    \foreach \i in {-3,...,2}
      \foreach \j in {-2,...,2}{
        \filldraw[black] (.5*\i,.5*\j) circle(.5pt);
      };
     
     \filldraw[black, thick] (.5*0,.5*1.5) circle(1pt);
     \filldraw[black, thick] (-.5*2.25,.5*0) circle(1pt);
     \filldraw[black, thick]
     (.5*0,.5*1) circle(1pt);

    \draw (-.5*2.25,0) node [above] {\tiny $(\lambda',0)$};
    \draw (0, .5*1.5) node [left] {\tiny $(0,-\frac{\lambda'}{b'})$};
    \draw (0, .5*1) node[below right=-2]
    {\tiny $(0,1)$};
    \draw (0,-.5*2) node[below]
    {\footnotesize Case $-\lambda' \geq b' > 0$};

 \end{scope}
\end{tikzpicture}
\hspace{2cm}
\begin{tikzpicture}\begin{scope}[scale=1.6]
     \filldraw[fill=gray!50,draw=none] (-.5*3,.5*.5)--(-.5*2.25,0)--(0,-.5*1.5)--(.5*0.75,-.5*2)--(.5*2,-.5*2)--(.5*2,.5*2)--(-.5*3,.5*2)--(-.5*3,.5*.5)
     
    ;
     \draw[gray,<->] (-1.5,0)--(1,0);
     \draw[gray,<->] (0,-1)--(0,1);
    \draw (-.5*3,.5*.5)--(-.5*2.25,0)--(0,-.5*1.5)--(.5*0.75,-.5*2)
    ;

    \foreach \i in {-2,...,2}
      \foreach \j in {-2,...,2}{
        \filldraw[black] (.5*\i,.5*\j) circle(.5pt);
      };
     
     \filldraw[black, thick] (.5*0,-.5*1.5) circle(1pt);
     \filldraw[black, thick] (-.5*2.25,.5*0) circle(1pt);
     \filldraw[black, thick]
     (.5*0,-.5*1) circle(1pt);

    \draw (-.5*2.25,0) node [above] {\tiny $(\lambda',0)$};
    \draw (0, -.5*1.5) node [left] {\tiny $(0,-\frac{\lambda'}{b'})$};
    \draw (0, -.5*1) node[above right=-2]
    {\tiny $(0,-1)$};
    \draw (0,-.5*2) node[below]
    {\footnotesize Case $\lambda' \leq b' < 0$};

 \end{scope}
\end{tikzpicture}
 \caption{We illustrate $\pi_1(\HH_{p,\lambda})$ for the cases when $c>0$, divided into cases according to whether $b'$ is $0$, $>0$ or $<0$.  }
    \label{fig: c>0 cases}
\end{figure}
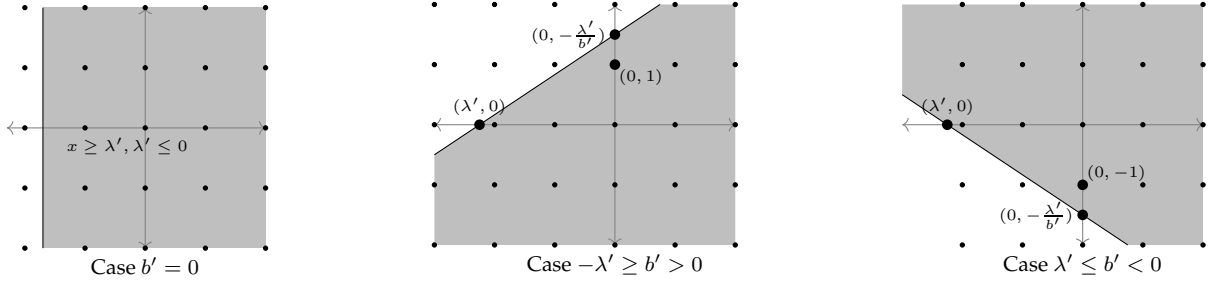

We have noted in Figure~\ref{fig: c>0 cases} that $\lvert \frac{\lambda'}{b'} \rvert \geq 1$ due to the relation between $\lambda'$ and $b'$. Moreover, there are choices of $a,b,c,\lambda$ such that we can obtain any value $\lambda' \leq 0, -\lambda'/b' \geq 1$.  

Next we consider the case $c<0$. Then 
$$
p_1(x,y) = p \circ \pi_1^{-1}(x,y)  = 
\begin{cases} 
cx-by, \quad \textup{ if } y \leq 0 \\
cx-by+cy, \quad \textup{ if } y \geq 0. 
\end{cases} 
$$
  For $\lambda \in \Z$ and $p_1$ as above, a computation similar to those above shows that the condition $S \subset \HH_{p,\lambda}$ is equivalent to the conditions $\lambda \leq 0, \lambda \leq b, \lambda \leq c-b$. Since the computations are similar, we omit details and record the results in Figure~\ref{fig: c<0 cases part 1} (the cases $b=0$ and $b>0,c<0$) and Figure~\ref{fig: c<0 cases part 2} (the cases $b=c<0$, $c<b<0$, and $b<c<0$).

  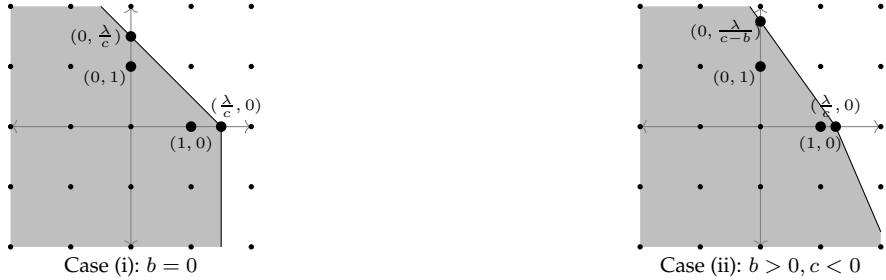
\begin{figure}[h]
\centering
\begin{minipage}{.5\textwidth}
  \centering
\begin{tikzpicture}\begin{scope}[scale=1.6]
     \filldraw[fill=gray!50,draw=none] 
     (-.5*2,.5*2)--(-.5*.5,.5*2)--(.5*1.5,-.5*0)--(.5*1.5,-.5*2)--(-.5*2,-.5*2)--(-.5*2,.5*2)
    ;
     \draw[gray,<->] (-1,0)--(1,0);
     \draw[gray,<->] (0,-1)--(0,1);
     \draw (-.5*.5,.5*2)--(.5*1.5,0);
     \draw (.5*1.5,0)--(.5*1.5,-.5*2);
    \foreach \i in {-2,...,2}
      \foreach \j in {-2,...,2}{
        \filldraw[black] (.5*\i,.5*\j) circle(.5pt);
      };
     \filldraw[black, thick] (.5*0,.5*1) circle(1pt);
     \filldraw[black, thick] (.5*1,.5*0) circle(1pt);
     \filldraw[black, thick] (.5*0,.5*1.5) circle(1pt);
     \filldraw[black, thick] (.5*1.5,.5*0) circle(1pt);

 \draw (0,.5*1) node [below left=-2] {\tiny $(0,1)$};
    \draw (.5*1, 0) node [below] {\tiny $(1,0)$};
     \draw (0,.5*1.5) node [left] {\tiny $(0,\frac{\lambda}{c})$};
    \draw (.5*1.75, 0) node [above] {\tiny $(\frac{\lambda}{c},0)$};
    \draw (0,-.5*2) node [below] 
    {\footnotesize Case (i): $b=0$};
\end{scope}\end{tikzpicture}

\end{minipage}%
\begin{minipage}{.5\textwidth}
\centering
 \begin{tikzpicture}\begin{scope}[scale=1.6]
     \filldraw[fill=gray!50,draw=none] 
     (-.5*2,.5*2)--(-.5*0.75,.5*2)--(0,.5*1.4)--(.5*1.75,0)--(.5*2,-.5*1)--(.5*2,-.5*2)--(-.5*2,-.5*2)--(-.5*2,.5*2)
    ;
     \draw[gray,<->] (-1,0)--(1,0);
     \draw[gray,<->] (0,-1)--(0,1);
     \draw (-.5*0.75,.5*2)--(0,.5*1.4)--(.5*1.75,0)--(.5*2,-.5*1);
    \foreach \i in {-2,...,2}
      \foreach \j in {-2,...,2}{
        \filldraw[black] (.5*\i,.5*\j) circle(.5pt);
      };
     \filldraw[black, thick] (.5*0,.5*1) circle(1pt);
     \filldraw[black, thick] (.5*1,.5*0) circle(1pt);
     \filldraw[black, thick] (.5*0,.5*1.4) circle(1pt);
     \filldraw[black, thick] (.5*1.75,.5*0) circle(1pt);

 \draw (0,.5*1) node [below left=-2] {\tiny $(0,1)$};
    \draw (.5*1, 0) node [below] {\tiny $(1,0)$};
     \draw (0,.5*1.4) node [above right=-3] {\tiny $(0,\frac{\lambda}{c-b})$};
    \draw (.5*1.75, 0) node [above right=-2] {\tiny $(\frac{\lambda}{c},0)$};
    \draw (0,-.5*2) node [below] 
    {\footnotesize Case (ii): $b>0, c<0$};
\end{scope}\end{tikzpicture}
\end{minipage}

 \caption{Case (i): on the left we show the $b=0$ case. Case (ii): on the right we illustrate the case $b>0,c<0$, in which we have $\lambda \leq c-b <c<0<b$. For $x\leq \frac{\lambda}{c}$, the boundary is defined by the equation $x=(\frac{b}{c}-1)y+\frac{\lambda}{c}$. For $x \geq \frac{\lambda}{c}$, the boundary is given by $x=\frac{b}{c}y+\frac{\lambda}{c}$. Note that $0<\frac{\lambda}{c-b}<\frac{\lambda}{c}$ since $c-b<c<0$. }
    \label{fig: c<0 cases part 1}
\end{figure}

 \begin{figure}[h]
\centering
\begin{minipage}{.3\textwidth}
  \centering

\begin{tikzpicture}\begin{scope}[scale=1.6]
     
     \filldraw[fill=gray!50,draw=none]
     (.5*1.5,.5*2)--(.5*1.5,0)--(0,-.5*1.5)--(-.5*.5,-.5*2)--(-.5*2,-.5*2)--(-.5*2,.5*2)--(.5*1.5,.5*2)
     ;
     \draw[gray,<->] (-1,0)--(1,0);
     \draw[gray,<->] (0,-1)--(0,1);
     \draw (.5*1.5,0)--(.5*1.5,.5*2);
     \draw (.5*1.5,0)--(0,-.5*1.5)--(-.5*.5,-.5*2);
    
    \foreach \i in {-2,...,2}
      \foreach \j in {-2,...,2}{
        \filldraw[black] (.5*\i,.5*\j) circle(.5pt);
      };
     \filldraw[black, thick] (.5*0,-.5*1) circle(1pt);
     \filldraw[black, thick] (.5*1,.5*0) circle(1pt);
     \filldraw[black, thick] (.5*0,-.5*1.5) circle(1pt);
     \filldraw[black, thick] (.5*1.5,.5*0) circle(1pt);

 \draw (0,-.5*1) node [above left=-2] {\tiny $(0,-1)$};
    \draw (.5*1, 0) node [below] {\tiny $(1,0)$};
     \draw (0,-.5*1.5) node [left] {\tiny $(0,-\frac{\lambda}{c})$};
    \draw (.5*1.5, 0) node [above] {\tiny $(\frac{\lambda}{c},0)$};
    \draw (0,-.5*2) node [below] 
    {\footnotesize Case (iii): $b=c<0$};
\end{scope}
\end{tikzpicture}

\end{minipage}%
\begin{minipage}{.3\textwidth}
\centering

\begin{tikzpicture}\begin{scope}[scale=1.6]
     
      \filldraw[fill=gray!50,draw=none]
      (-.5*0.17857,.5*2)--(0,.5*1.75)--(.5*1.25,0)--(-.5*.6,-.5*2)--(-.5*2,-.5*2)--(-.5*2,.5*2)--(-.5*0.17857,.5*2)
      ;
     \draw[gray,<->] (-1,0)--(1,0);
     \draw[gray,<->] (0,-1)--(0,1);
      \draw (-.5*0.17857,.5*2)--(0,.5*1.75)--(.5*1.25,0);
      \draw
      (.5*1.25,0)--(-.5*.6,-.5*2);
      
    \foreach \i in {-2,...,2}
      \foreach \j in {-2,...,2}{
        \filldraw[black] (.5*\i,.5*\j) circle(.5pt);
      };

  \filldraw[black, thick] (.5*0,.5*1) circle(1pt);
     \filldraw[black, thick] (.5*.5,.5*0) circle(1pt);
     \filldraw[black, thick] (.5*0,.5*1.75) circle(1pt);
     \filldraw[black, thick] (.5*1.25,.5*0) circle(1pt);
      
      \filldraw[black, thick] (.5*0,-.5*1) circle(1pt);
      \filldraw[black, thick] (.5*0,-.5*1.5) circle(1pt);

 \draw (0,.5*1) node [below left=-2] {\tiny $(0,1)$};
    \draw (.5*.5, 0) node [below] {\tiny $(0.5,0)$};
     \draw (0,.5*1.75) node [below left=-4] {\tiny $(0,\frac{\lambda}{c-b})$};
    \draw (.5*1.25, 0) node [above] {\tiny $(\frac{\lambda}{c},0)$};
    \draw (0,-.5*2) node [below] 
    {\footnotesize Case (iv): $c<b<0$}
    ;

 \draw (0,-.5*1) node [above left=-2] {\tiny $(0,-1)$};
     \draw (0,-.5*1.5) node [left] {\tiny $(0,-\frac{\lambda}{b})$};
\end{scope}
\end{tikzpicture}

\end{minipage}%
\begin{minipage}{.3\textwidth}
\centering

\begin{tikzpicture}\begin{scope}[scale=1.6]
      \filldraw[fill=gray!50,draw=none]
      (.5*2,.5*2)--(.5*2,.5*1.8)--(.5*1.25,0)--(-.5*.6,-.5*2)--(-.5*2,-.5*2)--(-.5*2,.5*2)--(-.5*2,.5*2)--(.5*2,.5*2)
      ;
     \draw[gray,<->] (-1,0)--(1,0);
     \draw[gray,<->] (0,-1)--(0,1);
      \draw
      (.5*1.25,0)--(.5*2,.5*1.8);
      \draw
      (.5*1.25,0)--(-.5*.6,-.5*2);

    \foreach \i in {-2,...,2}
      \foreach \j in {-2,...,2}{
        \filldraw[black] (.5*\i,.5*\j) circle(.5pt);
      };

     \filldraw[black, thick] (.5*1,.5*0) circle(1pt);

     \filldraw[black, thick] (.5*1.25,.5*0) circle(1pt);
      
      \filldraw[black, thick] (.5*0,-.5*1) circle(1pt);
      \filldraw[black, thick] (.5*0,-.5*1.5) circle(1pt);

    \draw (.5*1, 0) node [below] {\tiny $(1,0)$};
  
    \draw (.5*1.25, 0) node [above] {\tiny $(\frac{\lambda}{c},0)$};
    \draw (0,-.5*2) node [below] 
    {\footnotesize Case (v): $b<c<0$}
    ;

 \draw (0,-.5*1) node [above left=-2] {\tiny $(0,-1)$};
     \draw (0,-.5*1.5) node [left] {\tiny $(0,-\frac{\lambda}{b})$};
\end{scope}
\end{tikzpicture}

\end{minipage}
 \caption{Case (iii): $b=c<0$. Here we know $\lambda/c \geq 1$. Case (iv): $c<b<0$. In this case we know $\lambda/c\geq 1/2$. Case (v): $b<c<0$. In the region $x \leq \lambda/c$, the boundary is given by a linear function of slope $\frac{c}{b}$, while on the region $x \geq \lambda/c$, the boundary is given by a linear function of slope $\frac{c}{b-c} > \frac{c}{b}$. }
    \label{fig: c<0 cases part 2}
\end{figure}
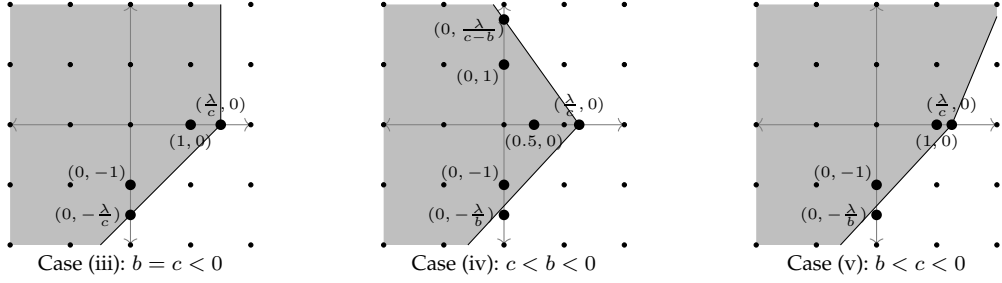

For case (iv) in Figure~\ref{fig: c<0 cases part 2} we remark that the lower bound $\lambda/c \geq 1/2$ is found by observing that since $\lambda \leq b$ and $\lambda \leq c-b$, we know $\frac{\lambda+b}{c} \geq 1$ and hence $\lambda/c \geq 1 - b/c$. On the other hand we also know $\lambda \leq b$ so $\lambda/c \geq b/c$, so $\lambda/c \geq \max\{1-b/c, b/c\}$. By the hypotheses we know $b/c\geq 0$ so $\lambda/c$ must be larger than the min of the function $\max\{1-x,x\}$ on $[0,\infty)$ which is $1/2$. It is possible to achieve this min by selecting $2b=c$ and $\lambda=c/2$. 

The above computations determine the set of $p,\lambda$ with $S \subset \HH_{p,\lambda}$ and it now follows that the image under $\pi_1$ of the intersection of all such $\HH_{p,\lambda}$ is as depicted on the left in Figure~\ref{fig: pt conv hull S}. By mutating, it is also immediate that $\pi_1(\ptconv_\R(S))$ is given by the figure on the right. 

\begin{figure}[h]
\centering

\begin{tikzpicture}\begin{scope}[scale=1.3]
     \filldraw[fill=gray!50,draw=none] 
     (-.5*0,.5*1)--(.5*0.5,0)--(.5*0,-.5*1)--(.5*0,-.5*0)
    ;
     \draw[gray,<->] (-.7,0)--(.7,0);
     \draw[gray,<->] (0,-.7)--(0,.7);
    \draw (-.5*0,.5*1)--(.5*0.5,0)--(.5*0,-.5*1)--(.5*0,-.5*0)--(-.5*0,.5*1)
    ;
    \foreach \i in {-1,...,1}
      \foreach \j in {-1,...,1}{
        \filldraw[black] (.5*\i,.5*\j) circle(.5pt);
      };
\end{scope}\end{tikzpicture}
\hspace{2cm}
\begin{tikzpicture}\begin{scope}[scale=1.3]
     \filldraw[fill=gray!50,draw=none] 
     (-.5*0,.5*1)--(.5*0,0)--(-.5*1,-.5*1)--(.5*0,.5*1)
    ;
    \draw (-.5*0,.5*1)--(.5*0,0)--(-.5*1,-.5*1)--(.5*0,.5*1)
    ;
     \draw[gray,<->] (-.7,0)--(.7,0);
     \draw[gray,<->] (0,-.7)--(0,.7);
    \foreach \i in {-1,...,1}
      \foreach \j in {-1,...,1}{
        \filldraw[black] (.5*\i,.5*\j) circle(.5pt);
      };
\end{scope}\end{tikzpicture}

 \caption{The image of $\ptconv_\R(S)$ in the two charts, with $\pi_1(\ptconv_\R(S))$ on the left and $\pi_2(\ptconv_\R(S))$ on the right.}
    \label{fig: pt conv hull S}
\end{figure}
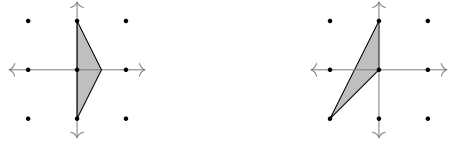

Thus we see that the point-convex hull of the set $S$ with $\pi_1(S)$ equal to $3$ collinear points in $M_1\otimes \R$ is a $2$-dimensional triangle, and in particular is not equal to the classical convex hull of $\pi_1(S)$, which is the $1$-dimensional line segment connecting the points in $\pi_1(S)$. However, the image of $S$ under $\pi_2$ is not collinear, and $\pi_2(\ptconv_\R(S))$ is in fact equal to the classical convex hull of $\pi_2(S)$.

\section{Examples: chart-Gorenstein-Fano polytopes in $\MM_s$}\label{sec: GF polytope}

 We now build, by way of example, several PL polytopes in $(\MM_s)_{\R}$ which are chart-Gorenstein-Fano in the sense of Definition~\ref{definition: PL Gorenstein Fano}. Since $\MM_s$ is rank $2$, by \cite[Lemma 5.21]{EscobarHaradaManon-PL} we expect the coordinate chart images of such a PL polytope to be classical $2$-dimensional Gorenstein-Fano polytopes. As mentioned in the introduction, we expect these examples to be related to past work e.g. of Petracci, Ilten, and Christophersen on deformations of toric varieties and complexity-$1$ $T$-varieties.

We begin with an example for $s=1$, where we give full details. 
We have seen from Proposition~\ref{proposition: points of Sp MMs} that a point in $\Sp(\MM_s)$ is specified by $3$ parameters, namely $(a,b,c) \in \mathcal{T}_s$ such that $a+b = s \cdot \min\{0,c\}$. Under this identification, our $3$ points $\mathsf{p},\mathsf{q},\mathsf{r}$ are specified by the choices 
$$
\mathsf{p} = (-2,2,1), \quad \mathsf{q} = (0,-1,-1), \quad \mathsf{r} = (1,-1,1)
$$
where the triples are interpreted as elements of $\mathcal{T}_s$. 
More concretely, this means that, for example, the point $\mathsf{p}$ expressed in $M_1$ coordinates $(x,y)$ and $M_2$ coordinates $(u,v)$ respectively, is 
$$
\mathsf{p}(x,y) = x-2y, \quad \mathsf{p}(u,v) = \begin{cases} -u-2v \quad \textup{ if } \, v \geq 0 \\ 
-u - v \quad \textup{ if } \, v < 0.
\end{cases}
$$
Note that $\mathsf{p}$ is linear on $M_1$.  The associated PL half-space $\mathcal{H}_{\mathsf{p},-1}$ is depicted in both charts in Figure \ref{fig_H_p}. 

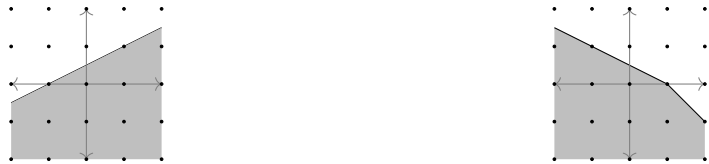
\begin{figure}[h]
    \centering
    \begin{tikzpicture}
    \draw (-.5*2,-.5*.5)--(.5*2,.5*1.5);
    \filldraw[fill=gray!50,draw=none] (-.5*2,-.5*.5)--(.5*2,.5*1.5)--(.5*2,-.5*2)--(-.5*2,-.5*2)--(-.5*2,-.5*.5);
    \draw[gray,<->] (-1,0)--(1,0);
    \draw[gray,<->] (0,-1)--(0,1);
    \foreach \i in {-2,...,2}
      \foreach \j in {-2,...,2}{
        \filldraw[black] (.5*\i,.5*\j) circle(.5pt);
      };
    \end{tikzpicture}
    \hspace{5cm}
    \begin{tikzpicture}
    \filldraw[fill=gray!50,draw=none] 
     (-.5*2,.5*1.5)--(.5*1,0)--(.5*2,-.5*1)--(.5*2,-.5*2)--(-.5*2,-.5*2)
    ;
    \draw (-.5*2,.5*1.5)--(.5*1,0)--(.5*2,-.5*1)
    ;
    \draw[gray,<->] (-1,0)--(1,0);
    \draw[gray,<->] (0,-1)--(0,1);
    \foreach \i in {-2,...,2}
      \foreach \j in {-2,...,2}{
        \filldraw[black] (.5*\i,.5*\j) circle(.5pt);
      };
    \end{tikzpicture}
    \caption{The two chart images of the PL half-space $\mathcal{H}_{\mathsf{p},-1}$. On the left is $\pi_1(\mathcal{H}_{\mathsf{p},-1})$ and on the right is $\pi_2(\mathcal{H}_{\mathsf{p},-1})$. }
    \label{fig_H_p}
\end{figure}

Similarly, the point $\mathsf{q}$ expressed in $M_1$ and $M_2$ coordinates is 
$$
\mathsf{q}(x,y) = \begin{cases} -x \quad \quad \textup{ if } \,  y \geq 0 \\ -x+y \quad \textup{ if } y \leq 0 \end{cases}
\quad \quad  \quad\textup{ and } \quad \quad \quad   \mathsf{q}(u,v) = u 
$$
So $\mathsf{q}$ is linear on $M_2$ and not linear on $M_1$.   The point $\mathsf{r}$ is given by 
$$
\mathsf{r}(x,y) = x+y, \quad \mathsf{r}(u,v) = \begin{cases} -u+v, \quad \textup{ if } \, v \geq 0 \\ 
-u + 2v, \quad \textup{ if } v \leq 0 \end{cases} 
$$
so $\mathsf{r}$ is linear on $M_1$. 
The associated PL half-spaces $\mathcal{H}_{\mathsf{q},-1}$ and $\mathcal{H}_{\mathsf{r},-1}$ are depicted, in both charts, in Figure \ref{fig_H_q} and Figure \ref{fig_H_r} respectively. 

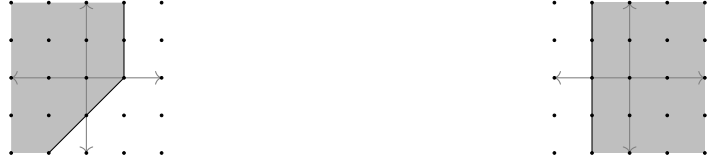
\begin{figure}[h]
    \centering
    \begin{tikzpicture}
    \filldraw[fill=gray!50,draw=none] (-.5*2,-.5*2)--(-.5*1,-.5*2)--(.5*0,-.5*1)--(.5*1,.5*0)--(.5*1,.5*1)--(.5*1,.5*2)--(-.5*2,.5*2)--(-.5*2,-.5*2);
    \draw  (-.5*1,-.5*2)--(.5*0,-.5*1)--(.5*1,.5*0)--(.5*1,.5*1)--(.5*1,.5*2);
    \draw[gray,<->] (-1,0)--(1,0);
    \draw[gray,<->] (0,-1)--(0,1);
    \foreach \i in {-2,...,2}
      \foreach \j in {-2,...,2}{
        \filldraw[black] (.5*\i,.5*\j) circle(.5pt);
      };
    \end{tikzpicture}
    \hspace{5cm}
    \begin{tikzpicture}
    \filldraw[fill=gray!50,draw=none] (-.5*1,-.5*2)--(-.5*1,.5*2)--(.5*2,.5*2)--(.5*2,-.5*2)--(-.5*1,-.5*2)
    ;
    \draw (-.5*1,-.5*2)--(-.5*1,.5*2)
    ;
    \draw[gray,<->] (-1,0)--(1,0);
    \draw[gray,<->] (0,-1)--(0,1);
    \foreach \i in {-2,...,2}
      \foreach \j in {-2,...,2}{
        \filldraw[black] (.5*\i,.5*\j) circle(.5pt);
      };
    \end{tikzpicture}
    \caption{The two chart images of the PL half-space $\mathcal{H}_{\mathsf{q},-1}$. On the left is $\pi_1(\mathcal{H}_{\mathsf{q},-1})$ and on the right is $\pi_2(\mathcal{H}_{\mathsf{q},-1})$.} 
    \label{fig_H_q}
\end{figure}

\begin{figure}[h]
    \centering
    \begin{tikzpicture}
     \filldraw[fill=gray!50,draw=none] (-.5*2,.5*1)--(.5*1,-.5*2)--(.5*2,-.5*2)--(.5*2,.5*2)--(-.5*2,.5*2)--(-.5*2,.5*1);
     \draw (-.5*2,.5*1)--(.5*1,-.5*2);
    \draw[gray,<->] (-1,0)--(1,0);
    \draw[gray,<->] (0,-1)--(0,1);
    \foreach \i in {-2,...,2}
      \foreach \j in {-2,...,2}{
        \filldraw[black] (.5*\i,.5*\j) circle(.5pt);
      };
    \end{tikzpicture}
    \hspace{5cm}
    \begin{tikzpicture}
    \filldraw[fill=gray!50,draw=none] (-.5*2,-.5*1.5)--(-.5*2,.5*2)--(.5*2,.5*2)--(.5*2,.5*1)--(.5*1,-.5*0)--(-.5*2,-.5*1.5);
    \draw (.5*2,.5*1)--(.5*1,-.5*0)--(-.5*2,-.5*1.5);
    \draw[gray,<->] (-1,0)--(1,0);
    \draw[gray,<->] (0,-1)--(0,1);
    \foreach \i in {-2,...,2}
      \foreach \j in {-2,...,2}{
        \filldraw[black] (.5*\i,.5*\j) circle(.5pt);
      };
    \end{tikzpicture}
    \caption{The two chart images of the PL half-space $\mathcal{H}_{\mathsf{r},-1}$. On the left is $\pi_1(\mathcal{H}_{\mathsf{r},-1})$ and on the right is $\pi_2(\mathcal{H}_{\mathsf{r},-1})$. } 
    \label{fig_H_r}
\end{figure}
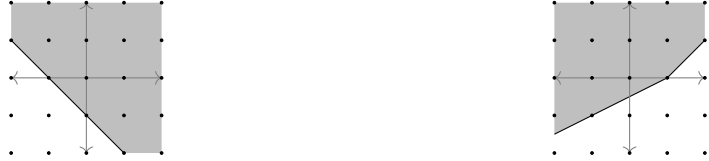

The intersection of these $3$ PL half-spaces is then a PL polytope $\PP$.  We depict both coordinate chart images of $\PP$ in Figure \ref{fig_charts-GF_polytopes_Ms}. We note that the two coordinate chart images $\pi_1(\PP)$ and $\pi_2(\PP)$ are both Gorenstein-Fano polytopes in the classical sense. Indeed, in the list of all $2$-dimensional Gorenstein-Fano polytopes (up to lattice isomorphism) given in \cite[p. 382]{Cox_Little_Schenck}, the two polytopes are of type 4b and 4c respectively. These are related by the piecewise-linear mutation $\mu_{1,2}: M_1 \to M_2$. 

\begin{figure}[h]
    \centering
    \begin{tikzpicture}
    \filldraw[fill=gray!50] (.5*0,-.5*1)--(.5*1,.5*0)--(.5*1,.5*1)--(-.5*1,.5*0)--(.5*0,-.5*1);
    \draw[gray,<->] (-1,0)--(1,0);
    \draw[gray,<->] (0,-1)--(0,1);
    \foreach \i in {-2,...,2}
      \foreach \j in {-2,...,2}{
        \filldraw[black] (.5*\i,.5*\j) circle(.5pt);
      };
    \end{tikzpicture}
    \hspace{5cm}
    \begin{tikzpicture}
    \filldraw[fill=gray!50] (-.5*1,-.5*1)--(-.5*1,.5*1)--(.5*1,.5*0)--(-.5*1,-.5*1)
    ;
    \draw[gray,<->] (-1,0)--(1,0);
    \draw[gray,<->] (0,-1)--(0,1);
    \foreach \i in {-2,...,2}
      \foreach \j in {-2,...,2}{
        \filldraw[black] (.5*\i,.5*\j) circle(.5pt);
      };
    \end{tikzpicture}
    \caption{The two chart images of the PL polytope $\PP = \mathcal{H}_{\mathsf{p},-1} \cap \mathcal{H}_{\mathsf{q},-1} \cap \mathcal{H}_{\mathsf{r},-1}$. On the left is $\pi_1(\PP)$ and on the right is $\pi_2(\PP)$. Both chart images are Gorenstein-Fano polytopes in the classical sense.}
    \label{fig_charts-GF_polytopes_Ms}
\end{figure}
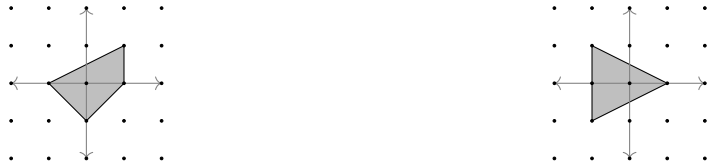

In the classical setting, the dual of a Gorenstein-Fano polytope is again a lattice polytope. We now explicitly compute $\PP^\vee$ for the example above to see that this is also the case in this specific example. Note that since $\MM_s$ is self-dual, both $\PP$ and $\PP^\vee$ are PL polytopes in $(\MM_s)_{\R}$. 
A computationally effective method of computing $\PP^\vee$ is given in \cite[Lemma 5.16]{EscobarHaradaManon-PL} which states 
\begin{equation}\label{eq: compute P dual}
\PP^\vee = \bigcap_{m \in \mathrm{V}(\PP)} \mathcal{H}_{\mathbf{v}(m), -1}
\end{equation}
where $\mathrm{V}(\PP)$ denotes the set of vertices of $\PP$ as in~\eqref{eq: def vertices of P} and $\mathbf{v}$ is the strict dual pairing. 
To take advantage of this characterization, we need the vertices of $\PP$ in our example. It is straightforward to compute 
$$
\mathrm{V}(\PP) = \{((1,1),(-1,1)), ((1,0),(-1,0)), ((0,-1),(-1,-1)), ((-1,0),(1,0)) \}.
$$
In order to interpret the vertices as points in $\Sp(\MM_s)$, we must take their images under the strict dual pairing $\mathbf{w}_s$ of $\MM_s$, as given in~\eqref{eq: def ws}, which takes values in $\mathcal{T}_s \cong \Sp(\MM_s)$.  We can then use these images to define the PL half-spaces in the RHS of~\eqref{eq: compute P dual}. The relevant data is summarized in the table below. 

\begin{tabular}{|c|c|c|c|c|c|}
\hline 
$m \in \mathrm{V}(\PP)$ & $\w(m)(\mathfrak{e}_2)$ & $\w(m)(\mathfrak{e}'_2)$ & $\w(m)(\mathfrak{e}_1)$ & $\w(m)$ on $M_1$ & $\w(m)$ on $M_2$   \\  \hline
$((1,1),(-1,1))$ & $1$ & $-1$ & $1$ & $x+y$  &$ \begin{cases} -u + v \, \textup{ if }  v \geq 0 \\ -u + 2v \, \textup{ if } v \leq 0 \end{cases}$  \\  \hline 
$((1,0),(-1,0))$  & $1$ &$-1$  & $0$ & $y$  & $v$  \\  \hline
$((0,-1),(-1,-1))$  & $0$  & $-1$ & $-1$ & $\begin{cases} -x+y \,  \textup{ if } \, y \leq 0 \\ -x \,  \textup{ if } \, y \geq 0 \end{cases} $  & $u$  \\ \hline
$((-1,0),(1,0))$ & $-1$  & $1$  & $0$ & $-y$ & $-v$  \\ 
\hline
\end{tabular}

The intersection of the $4$ PL half-spaces $\mathcal{H}_{\mathbf{w}_s(m),-1}$ for the $4$ vertices in $\mathrm{V}(\PP)$ is depicted in both coordinate charts in Figure \ref{fig_dual_polytopes}. 
The two chart images are equivalent up to a transformation in $SL(2,\Z)$, hence are lattice-equivalent. In the list of 2-dimensional Gorenstein-Fano polytopes given in \cite{Cox_Little_Schenck}, these two polytopes correspond to type 7b.

\begin{figure}[h]
    \centering
    \begin{tikzpicture}
    \filldraw[fill=gray!50] (-.5*2,.5*1)--(.5*1,.5*1)--(.5*1,.5*0)--(.5*0,-.5*1)--(-.5*2,.5*1);
    \draw[gray,<->] (-1.5,0)--(1.5,0);
    \draw[gray,<->] (0,-1)--(0,1);
    \foreach \i in {-2,...,2}
      \foreach \j in {-2,...,2}{
        \filldraw[black] (.5*\i,.5*\j) circle(.5pt);
      };
    \end{tikzpicture}
    \hspace{5cm}
    \begin{tikzpicture}
    \filldraw[fill=gray!50] (-.5*1,-.5*1)--(-.5*1,.5*1)--(.5*2,.5*1)--(.5*1,0)--(-.5*1,-.5*1)
    ;
    \draw[gray,<->] (-1,0)--(1,0);
    \draw[gray,<->] (0,-1)--(0,1);
    \foreach \i in {-2,...,2}
      \foreach \j in {-2,...,2}{
        \filldraw[black] (.5*\i,.5*\j) circle(.5pt);
      };
    \end{tikzpicture}
    \caption{The two chart images of the PL dual polytope $\PP^\vee$. On the left is $\pi_1(\PP^\vee)$ and on the right is $\pi_2(\PP^\vee)$.}
    \label{fig_dual_polytopes}
\end{figure}
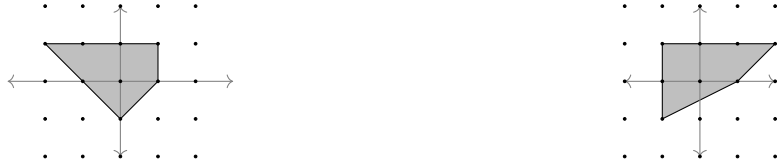

We now proceed to record several more examples of chart-Gorenstein-Fano polytopes in $\MM_s$. Since the computations are similar to those given above, we do not give details. We should emphasize here that we do not claim any general classification results or existence results; we leave this open for future exploration.

In the examples below, we record the representatives in $\mathcal{T}_s$ of the points with respect to which we define the half-spaces $\HH_{p,-1}$ whose intersection is the PL polytope $\PP$. We also illustrate the coordinate chart images of $\PP$ in $M_1$ and $M_2$, as well as the coordinate chart images of the dual polytope $\PP^\vee$. 

\begin{Example} 
We begin with another example in the $s=1$ case. We represent the five points in $\Sp(\MM_1)$ as elements in $\mathcal{T}_1$ using the identification given in Section~\ref{sec: MM_s}. With this understood, the points are 
$$
\mathsf{p} = (-1,0,-1), \quad \mathsf{q} = (1,-1,0), \quad \mathsf{r} = (-1,1,0), \quad \mathsf{s} = (0,0,1), \quad  \mathsf{t}= (1,-1,1).
$$
So the polytope $\PP$ is the intersection 
$$
\PP = \HH_{\mathsf{p},-1} \cap \HH_{\mathsf{q},-1} \cap \HH_{\mathsf{r},-1} \cap \HH_{\mathsf{s},-1} \cap \HH_{\mathsf{t},-1}
$$
where we have chosen all parameters $a$ in the definition of the half-spaces to be equal to $-1$, since we wish to describe a chart-Gorenstein-Fano PL polytope. The chart images of $\PP$ are as given in Figure~\ref{fig_s1_polytope}. 

\begin{figure}[h]
    \centering
    \begin{tikzpicture}
    \filldraw[fill=gray!50] (-.5*1,.5*1)--(.5*0,.5*1)--(.5*1,.5*0)--(.5*1,.5*-1)--(.5*0,-.5*1)--(-.5*1,.5*0)--(-.5*1,.5*1);
    \draw[gray,<->] (-1.5,0)--(1.5,0);
    \draw[gray,<->] (0,-1)--(0,1);
    \foreach \i in {-2,...,2}
      \foreach \j in {-2,...,2}{
        \filldraw[black] (.5*\i,.5*\j) circle(.5pt);
      };
    \end{tikzpicture}
    \hspace{5cm}
    \begin{tikzpicture}
    \filldraw[fill=gray!50] (.5*0,.5*1)--(.5*1,.5*1)--(.5*1,.5*0)--(-.5*1,-.5*1)--(-.5*1,-.5*1)--(-.5*2,-.5*1)--(.5*0,.5*1)
    ;
    \draw[gray,<->] (-1,0)--(1,0);
    \draw[gray,<->] (0,-1)--(0,1);
    \foreach \i in {-2,...,2}
      \foreach \j in {-2,...,2}{
        \filldraw[black] (.5*\i,.5*\j) circle(.5pt);
      };
    \end{tikzpicture}
    \caption{We illustrate the chart images of the chart-Gorenstein-Fano polytope $\PP$ defined by the five given points. The image $\pi_1(\PP)$ in $M_1 \otimes \R$ is illustrated on the left, $\pi_2(\PP) \subset M_2 \otimes \R$ is on the right.}
    \label{fig_s1_polytope}
\end{figure}
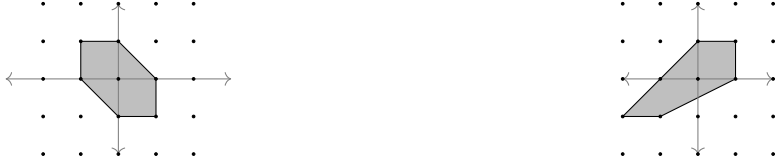

The vertices of this PL polytope can be seen to have chart image $\{(1,0),(0,1),(-1,1),(-1,0),(0,-1),(1,-1)\}$. Using the formula for $\w_s$ given in~\eqref{eq: def ws} we may then compute its associated points and the corresponding dual polytope $\PP^\vee$. 
We illustrate the resulting dual PL polytope $\PP^\vee$ in Figure~\ref{fig_s1_polytope_dual}.

\begin{figure}[h]
    \centering
    \begin{tikzpicture}
    \filldraw[fill=gray!50] (-.5*0,.5*1)--(.5*1,.5*1)--(.5*1,.5*0)--(-.5*1,.5*-1)--(-.5*1,-.5*0)--(-.5*0,.5*1);
    \draw[gray,<->] (-1.5,0)--(1.5,0);
    \draw[gray,<->] (0,-1)--(0,1);
    \foreach \i in {-2,...,2}
      \foreach \j in {-2,...,2}{
        \filldraw[black] (.5*\i,.5*\j) circle(.5pt);
      };
    \end{tikzpicture}
    \hspace{5cm}
    \begin{tikzpicture}
    \filldraw[fill=gray!50] (.5*0,.5*1)--(.5*1,.5*0)--(.5*0,.5*-1)--(-.5*0,-.5*1)--(-.5*1,-.5*0)--(-.5*1,.5*1)--(.5*0,.5*1)
    ;
    \draw[gray,<->] (-1,0)--(1,0);
    \draw[gray,<->] (0,-1)--(0,1);
    \foreach \i in {-2,...,2}
      \foreach \j in {-2,...,2}{
        \filldraw[black] (.5*\i,.5*\j) circle(.5pt);
      };
    \end{tikzpicture}
    \caption{The dual PL polytope $\PP^\vee$ of $\PP$. We depict its chart image in $M_1$ on the left, and the image in $M_2$ on the right.}
    \label{fig_s1_polytope_dual}
\end{figure}
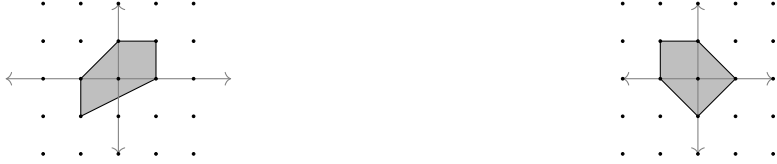

\end{Example}

We now begin an exploration of examples for the cases when $s>1$.

\begin{Example} 
In this example we take $s=2$, so the mutation is now $\mu_{1,2}(x,y)=(\min\{0,2y\}-x,y)$. In this case, the set $\mathcal{T}_2$ parametrizing the set of points $\Sp(\MM_2)$ is $\mathcal{T}_2\cap \Z^3 = \{(a,b,c) \in \R^3 \, \mid \, a+b = \min\{0,2c\}\}$.  We consider the following set of four points in $\mathcal{T}_2$:
$$
\mathsf{p} = (-1,-1,-1), \quad \mathsf{q} = (1,-1,1), \quad \mathsf{r} = (0,0,1), \quad 
\mathsf{s} = (-1,1,0).
$$
Then the chart-Gorenstein-Fano polytope $\PP$ is given by the half-spaces defined by the above points, with parameter $-1$. We illustrate its chart images in Figure~\ref{fig_s2_polytopes}. 

\begin{figure}[h]
    \centering
    \begin{tikzpicture}
    \filldraw[fill=gray!50] (.5*-1,.5*1)--(0,.5*1)--(.5*1,.5*0)--(-.5*0,-.5*1)--(.5*-1,.5*0)--(.5*-1,.5*1);
    \draw[gray,<->] (-1.5,0)--(1.5,0);
    \draw[gray,<->] (0,-1)--(0,1);
    \foreach \i in {-2,...,2}
      \foreach \j in {-2,...,2}{
        \filldraw[black] (.5*\i,.5*\j) circle(.5pt);
      };
    \end{tikzpicture}
    \hspace{5cm}
    \begin{tikzpicture}
    \filldraw[fill=gray!50] (.5*1,.5*1)--(.5*1,.5*0)--(-.5*2,-.5*1)--(0,.5*1)--(.5*1,.5*1)
    ;
    \draw[gray,<->] (-1,0)--(1,0);
    \draw[gray,<->] (0,-1)--(0,1);
    \foreach \i in {-2,...,2}
      \foreach \j in {-2,...,2}{
        \filldraw[black] (.5*\i,.5*\j) circle(.5pt);
      };
    \end{tikzpicture}
    \caption{Here we depict the chart-Gorenstein-Fano PL polytope $\PP$ in $\MM_2 \otimes \R$ given by the above points. The image $\pi_1(\PP)$ is on the left and $\pi_2(\PP)$ is on the right. }
    \label{fig_s2_polytopes}
\end{figure}
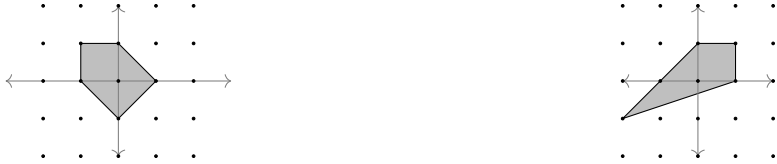

The set of vertices of this PL polytope can be seen to have image in $\pi_1$ equal to $\{(0,1),(1,0),(0,-1),(-1,0),(-1,1)\}$. Following the procedure already established we may compute the dual PL polytope $\PP^\vee$; we depict the result in Figure~\ref{fig_M2_dual_polytopes}. 

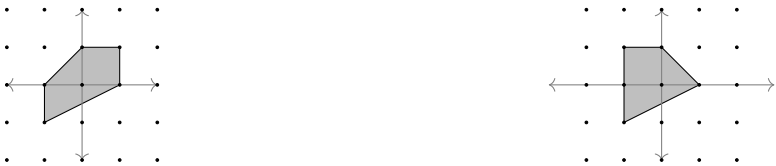
\begin{figure}[h]
    \centering
    
    \begin{tikzpicture}
    \filldraw[fill=gray!50] (.5*-1,.5*0)--(.5*0,.5*1)--(.5*1,.5*1)--(.5*1,.5*0)--(.5*-1,.5*-1)--(.5*-1,.5*0)
    ;
    \draw[gray,<->] (-1,0)--(1,0);
    \draw[gray,<->] (0,-1)--(0,1);
    \foreach \i in {-2,...,2}
      \foreach \j in {-2,...,2}{
        \filldraw[black] (.5*\i,.5*\j) circle(.5pt);
      };
    \end{tikzpicture}
    \hspace{5cm}
    \begin{tikzpicture}
    \filldraw[fill=gray!50] (.5*1,.5*0)--(.5*-1,.5*-1)--(.5*-1,.5*0)--(.5*-1,.5*1)--(.5*0,.5*1)--(.5*1,.5*0);
    \draw[gray,<->] (-1.5,0)--(1.5,0);
    \draw[gray,<->] (0,-1)--(0,1);
    \foreach \i in {-2,...,2}
      \foreach \j in {-2,...,2}{
        \filldraw[black] (.5*\i,.5*\j) circle(.5pt);
      };
    \end{tikzpicture}
    
    \caption{The two chart images of the dual PL polytope $\PP^\vee$ in $\MM_2 \otimes \R$, with the chart image in $M_1$ on the left and $M_2$ on the right.}
    \label{fig_M2_dual_polytopes}
\end{figure}

\end{Example}

\begin{Example} 
We continue with an example for $s=3$. Since the details are similar as for the previous cases, we will be brief. We choose points 
in $\mathcal{T}_3$ as follows: 
$$
\mathsf{p} = (-2,-1,-1), \quad \mathsf{q} = (1,-1,1), \quad \mathsf{r} = (0,0,1).
$$
The corresponding $\PP$ is depicted in Figure~\ref{fig_M3_polytopes}.

\begin{figure}[h]
    \centering
    \begin{tikzpicture}
    \filldraw[fill=gray!50] (.5*-1,.5*1)--(.5*1,0)--(.5*0,.5*-1)--(-.5*1,-.5*0)--(.5*-1,.5*1);
    \draw[gray,<->] (-1.5,0)--(1.5,0);
    \draw[gray,<->] (0,-1)--(0,1);
    \foreach \i in {-2,...,2}
      \foreach \j in {-2,...,2}{
        \filldraw[black] (.5*\i,.5*\j) circle(.5pt);
      };
    \end{tikzpicture}
    \hspace{5cm}
    \begin{tikzpicture}
    \filldraw[fill=gray!50] (.5*1,.5*1)--(.5*1,.5*0)--(-.5*3,-.5*1)--(.5*1,.5*1)
    ;
    \draw[gray,<->] (-1,0)--(1,0);
    \draw[gray,<->] (0,-1)--(0,1);
    \foreach \i in {-2,...,2}
      \foreach \j in {-2,...,2}{
        \filldraw[black] (.5*\i,.5*\j) circle(.5pt);
      };
    \end{tikzpicture}
    \caption{A chart-Gorenstein-Fano polytope $\PP$ in $\MM_3 \otimes \R$ corresponding to the given $3$ points. The $M_1$ image is on the left and $M_2$ on the right. }
    \label{fig_M3_polytopes}
\end{figure}
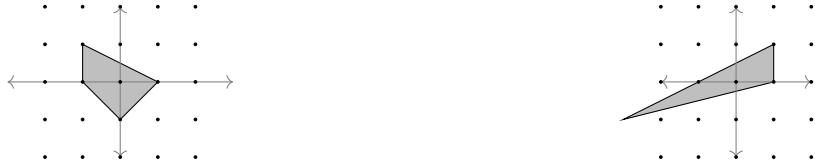

The vertices of $\PP$ have image in $M_1$ given by $\{(1,0),(-1,0),(-1,1),(0,-1)\}$. The dual PL polytope $\PP^\vee$ is depicted in Figure~\ref{fig_M3_dual_polytopes}.

\begin{figure}[h]
    \centering
    \begin{tikzpicture}
    \filldraw[fill=gray!50] (.5*0,.5*1)--(.5*1,.5*1)--(.5*1,-.5*0)--(.5*-2,.5*-1)--(.5*0,.5*1);
    \draw[gray,<->] (-1.5,0)--(1.5,0);
    \draw[gray,<->] (0,-1)--(0,1);
    \foreach \i in {-2,...,2}
      \foreach \j in {-2,...,2}{
        \filldraw[black] (.5*\i,.5*\j) circle(.5pt);
      };
    \end{tikzpicture}
    \hspace{5cm}
    \begin{tikzpicture}
    \filldraw[fill=gray!50] (.5*-1,.5*1)--(.5*0,.5*1)--(.5*1,-.5*0)--(.5*-1,.5*-1)--(.5*-1,.5*1)
    ;
    \draw[gray,<->] (-1,0)--(1,0);
    \draw[gray,<->] (0,-1)--(0,1);
    \foreach \i in {-2,...,2}
      \foreach \j in {-2,...,2}{
        \filldraw[black] (.5*\i,.5*\j) circle(.5pt);
      };
    \end{tikzpicture}
    \caption{The two chart images of the dual PL polytope $\PP^\vee$ in $\MM_3 \otimes \R$. The image in $M_1$ is on the left and $M_2$ on the right.}
    \label{fig_M3_dual_polytopes}
\end{figure}
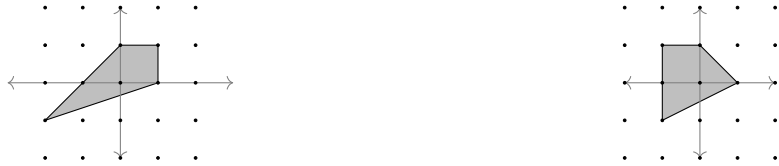
\end{Example}

\begin{Example} 
Finally, we give an example for the $s=4$ case. The points chosen are 
$$
\mathsf{p} = (-2,-2,-1), \quad \mathsf{q} = (0,0,1).
$$
It is useful to note that this exhibits different behavior of the PL situation from the classical one, since we may define a bounded PL polytope with only two PL half-spaces. The resulting PL polytope $\PP$ is shown in Figure~\ref{fig_M4_polytopes}. The vertices $V(\PP)$ have chart image in $M_1$ given by $\{(1,0),(-1,1),(-1,0),(-1,-1)\}$. The dual PL polytope $\PP^\vee$ is given in Figure~\ref{fig_M4_dual_polytopes}.

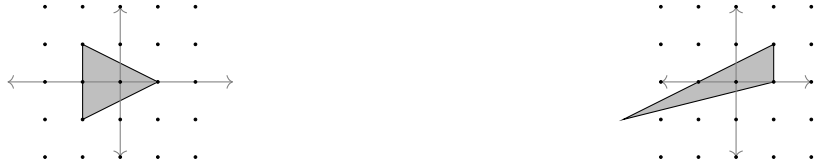
\begin{figure}[h]
    \centering
    \begin{tikzpicture}
    \filldraw[fill=gray!50] (-.5*1,.5*1)--(.5*1,.5*0)--(-.5*1,-.5*1)--(-.5*1,.5*1);
    \draw[gray,<->] (-1.5,0)--(1.5,0);
    \draw[gray,<->] (0,-1)--(0,1);
    \foreach \i in {-2,...,2}
      \foreach \j in {-2,...,2}{
        \filldraw[black] (.5*\i,.5*\j) circle(.5pt);
      };
    \end{tikzpicture}
    \hspace{5cm}
    \begin{tikzpicture}
    \filldraw[fill=gray!50] (.5*1,.5*1)--(.5*1,.5*0)--(-.5*3,-.5*1)--(.5*1,.5*1)
    ;
    \draw[gray,<->] (-1,0)--(1,0);
    \draw[gray,<->] (0,-1)--(0,1);
    \foreach \i in {-2,...,2}
      \foreach \j in {-2,...,2}{
        \filldraw[black] (.5*\i,.5*\j) circle(.5pt);
      };
    \end{tikzpicture}
    \caption{The two chart images of the chart-Gorenstein-Fano PL polytope $\PP$ in $\MM_4 \otimes \R$. The image in $M_1$ is on the left and $M_2$ on the right.}
    \label{fig_M4_polytopes}
\end{figure}

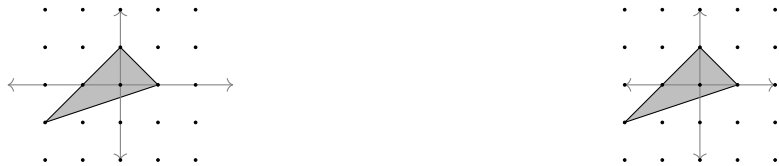
\begin{figure}[h]
    \centering
    \begin{tikzpicture}
    \filldraw[fill=gray!50] (.5*0,.5*1)--(.5*1,.5*0)--(-.5*2,-.5*1)--(.5*0,.5*1);
    \draw[gray,<->] (-1.5,0)--(1.5,0);
    \draw[gray,<->] (0,-1)--(0,1);
    \foreach \i in {-2,...,2}
      \foreach \j in {-2,...,2}{
        \filldraw[black] (.5*\i,.5*\j) circle(.5pt);
      };
    \end{tikzpicture}
    \hspace{5cm}
    \begin{tikzpicture}
    \filldraw[fill=gray!50] (.5*0,.5*1)--(.5*1,.5*0)--(-.5*2,-.5*1)--(.5*0,.5*1)
    ;
    \draw[gray,<->] (-1,0)--(1,0);
    \draw[gray,<->] (0,-1)--(0,1);
    \foreach \i in {-2,...,2}
      \foreach \j in {-2,...,2}{
        \filldraw[black] (.5*\i,.5*\j) circle(.5pt);
      };
    \end{tikzpicture}
    \caption{The two chart images of the dual PL polytope $\PP^\vee$. The chart image in $M_1$ is on the left and $M_2$ on the right, though in fact they are the same.}
    \label{fig_M4_dual_polytopes}
\end{figure}

\end{Example}

We finish with an example of a chart-Gorenstein-Fano PL polytope $\PP$ in the case $s=1$ that has the property that its dual PL polytope is \emph{not} an integral polytope. This example, together with \cite[Example 5.17]{EscobarHaradaManon-PL}, suggest that the convex geometry of dual PL polytopes is subtle.

\begin{Example}\label{example: dual not integral}
Here we choose $s=1$ so we are working in $\MM_1$. 
The points chosen are 
$$
\mathsf{p} = (0,0,1), \quad \mathsf{q} = (2,-2,1), \quad
\mathsf{r} = (-1,0,-1).
$$
As usual we choose all parameters $a_i$ defining the half-spaces to be equal to $-1$. 
The resulting chart-Gorenstein-Fano PL polytope $\PP$ is shown in Figure~\ref{fig_nonlattice_polytopes}. The vertices $V(\PP)$ have chart image in $M_1$ given by $\{(1,0),(1,-1),(-1,0),(-1,2)\}$. The dual PL polytope $\PP^\vee$ is given in Figure~\ref{fig_nonlattice_dual_polytopes}. As we can see from Figure~\ref{fig_nonlattice_dual_polytopes}, the dual PL polytope $\PP^\vee$ is \emph{not} an integral polytope due to the presence of the vertex $(0.5,0)$. 

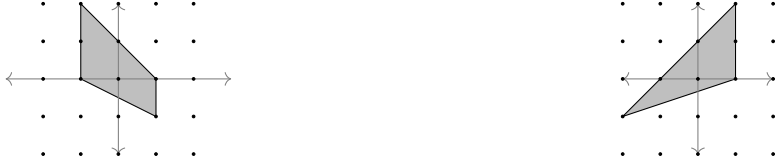
\begin{figure}[h]
    \centering
    \begin{tikzpicture}
    \filldraw[fill=gray!50] (-.5*1,.5*2)--(.5*1,.5*0)--(.5*1,-.5*1)--(-.5*1,.5*0)--(-.5*1,.5*2);
    \draw[gray,<->] (-1.5,0)--(1.5,0);
    \draw[gray,<->] (0,-1)--(0,1);
    \foreach \i in {-2,...,2}
      \foreach \j in {-2,...,2}{
        \filldraw[black] (.5*\i,.5*\j) circle(.5pt);
      };
    \end{tikzpicture}
    \hspace{5cm}
    \begin{tikzpicture}
    \filldraw[fill=gray!50] (.5*1,.5*0)--(-.5*2,-.5*1)--(.5*1,.5*2)--(.5*1,.5*0)
    ;
    \draw[gray,<->] (-1,0)--(1,0);
    \draw[gray,<->] (0,-1)--(0,1);
    \foreach \i in {-2,...,2}
      \foreach \j in {-2,...,2}{
        \filldraw[black] (.5*\i,.5*\j) circle(.5pt);
      };
    \end{tikzpicture}
    \caption{The two chart images of the chart-Gorenstein-Fano PL polytope $\PP$ in $\MM_1 \otimes \R$. The image in $M_1$ is on the left and $M_2$ on the right.}
    \label{fig_nonlattice_polytopes}
\end{figure}

\begin{figure}[h]
    \centering
    \begin{tikzpicture}
    \filldraw[fill=gray!50] (.5*0,.5*1)--(.5*2,.5*1)--(.5*1,.5*0)--(-.5*1,-.5*1)--(.5*0,.5*1);
    \draw[gray,<->] (-1.5,0)--(1.5,0);
    \draw[gray,<->] (0,-1)--(0,1);
    \foreach \i in {-3,...,3}
      \foreach \j in {-2,...,2}{
        \filldraw[black] (.5*\i,.5*\j) circle(.5pt);
      };
    \end{tikzpicture}
    \hspace{5cm}
    \begin{tikzpicture}
    \filldraw[fill=gray!50] (.5*0,.5*1)--(-.5*2,.5*1)--(.5*0,-.5*1)--(.5*0.5,.5*0)--(.5*0,.5*1);
    \draw[gray,<->] (-1.5,0)--(1.5,0);
    \draw[gray,<->] (0,-1)--(0,1);
    \foreach \i in {-3,...,3}
      \foreach \j in {-2,...,2}{
        \filldraw[black] (.5*\i,.5*\j) circle(.5pt);
      };
    \end{tikzpicture}
    \caption{The two chart images of the dual PL polytope $\PP^\vee$. The chart image in $M_1$ is on the left and $M_2$ on the right. Note that the chart image in $M_2$ is not an integral polytope.}
    \label{fig_nonlattice_dual_polytopes}
\end{figure}
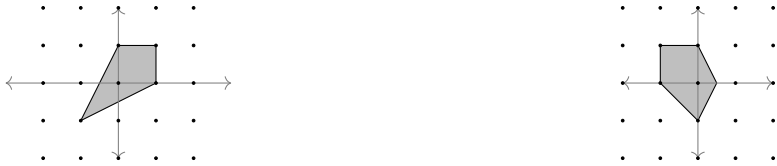

\end{Example}

\section{The polyptych lattice $\MM_s$ is detropicalizable}\label{sec: detrop MMs}

Having explored convex geometry in Sections~\ref{sec: pt conv hull example} and~\ref{sec: GF polytope}, we now return to the algebra and algebraic geometry. 
Our main goal in this section is to exhibit a detropicalization of the polyptych lattice $\MM_s$, thus showing that $\MM_s$ is detropicalizable. 
To accomplish this goal, we need some preliminaries.  

 We first identify the coordinate charts $M_1$ and $M_2$ of $\MM_s$ with sublattices of $\Z^2 \times \Z^2$. Define
$$
M_s(1) := \{(a_1,a_2,b_1,b_2) \in \Z^2 \times \Z^2 \, \mid \, a_2 = b_1 = 0\} = \{(a_1, 0,0,b_2)\}
$$
and 
$$
M_s(2) := \{(c_1,c_2,d_1,d_2) \in \Z^2 \times \Z^2 \, \mid \, c_2=d_2 =0\} = \{(c_1,0,d_1,0)\}.
$$
Now we choose identifications 
\begin{equation}\label{eq: M1 to Ms1}
\Theta_1: M_1 \to M_s(1), \quad (x,y) \mapsto (y,0,0,x) 
\end{equation}
and
\begin{equation}\label{eq: M2 to Ms2}
\Theta_2: M_2 \to M_s(2), \quad (u,v) \mapsto (v,0,u,0). 
\end{equation}
If we define a mutation map
\begin{equation}\label{eq: tilde mu 12}
\tilde{\mu}_{1,2}: M_s(1) \to M_s(2),
\quad (a,0,0,b) \mapsto 
(a, 0, \min\{0,s \cdot a\} - b, 0) 
\end{equation}
from $M_s(1) \to M_s(2)$ then it is straightforward to check that the following diagram commutes: 
$$\begin{tikzcd}
M_1 \arrow[r, "\mu_{1,2}"] \arrow[d,"\Theta_{1}" black]
& M_2 \arrow[d, "\Theta_{2}" black] \\
M_s(1) \arrow[r, black, "\tilde{\mu}_{1,2}" black]
& |[black]| M_s(2)
\end{tikzcd}$$
so we may realize $\MM_s$ in terms of the coordinate charts $M_s(1), M_s(2)$ in place of $M_1,M_2$. This will be convenient for some of our arguments below. We also specify a subset $\mathbb{M}_s$ of $\Z^2 \times \Z^2$ which we will identify with (the set of elements of) $\MM_s$ as follows: 
$$
\mathbb{M}_s := \left\{ (w_1,w_2,z_1,z_2) \in \Z^2 \times \Z^2 \, \mid \, \min\{w_1,w_2\} = 0, \, -s \cdot w_2 = z_1+z_2\right\}. 
$$
We then define maps (which we may think of as projections to coordinate charts) $\Psi_i: \mathbb{M}_s \to M_s(i)$ by 
\begin{equation}\label{eq: Psi1}
    \Psi_1(w_1,w_2,z_1,z_2) = (w_1-w_2, 0, 0, z_2)
\end{equation}
and 
\begin{equation}
    \Psi_2(w_1,w_2,z_1,z_2) = (w_1-w_2,0,z_1,0).
\end{equation}
It is straightforward to compute the inverse of $\Psi_1$ to be 
\begin{equation}\label{eq: Psi1 inverse}
\Psi_1^{-1}(a,0,0,b) = 
\begin{cases}
    (a,0,-b,b) \, \, \quad \textup{ if } a \geq 0 \\
    (0,-a,sa-b,b) \, \, \textup{ if } a \leq 0 
\end{cases}
\end{equation}
from which it immediately follows that 
$$
\Psi_2 \circ \Psi_1^{-1}(a,0,0,b) = (a,0,\min\{0,s \cdot a\} - b, 0).
$$
Note that this is the same as the mutation map $\tilde{\mu}_{1,2}$ of~\eqref{eq: tilde mu 12}.

Just as we have identified a subset of $\Z^2 \times \Z^2$ with $\MM_s$, we now define a subset of $\Z^2 \times \Z^2$ which corresponds to the space of points $\Sp(\MM_s)$ as follows. We define 
\begin{equation}\label{eq: def Ts}
\mathbb{T}_s := \{(\alpha_1,\alpha_2,\beta_1,\beta_2) \in \Z^2 \times \Z^2 \, \mid \, \alpha_1+\alpha_2 = s \cdot \min\{\beta_1,\beta_2\}, \, \, \beta_2 = 0\}.
\end{equation}
We also define 
\begin{equation}\label{eq: def Ts1}
\mathbb{T}_s(1) := \{(\alpha_1,\alpha_2,\beta_1,\beta_2) \in \mathbb{T}_s \, \mid \, \alpha_1+\alpha_2 = s \cdot \beta_1\} 
\end{equation}
and
\begin{equation}\label{eq: def Ts2}
\mathbb{T}_s(2) := \{(\alpha_1,\alpha_2,\beta_1,\beta_2) \in \mathbb{T}_s \, \mid \, \alpha_1+\alpha_2 = s \cdot \beta_2 = 0\}
\end{equation}
so $\mathbb{T}_s(1)$ consists of those tuples $(\alpha_1,\alpha_2,\beta_1,\beta_2) \in \mathbb{T}_s$ where $\beta_1 \leq \beta_2 =0$, whereas $\mathbb{T}_s(2)$ consists of the tuples where $\beta_2=0 \leq \beta_1$. 

Given a tuple $(\bar{\alpha},\bar{\beta}) = (\alpha_1,\alpha_2,\beta_1,\beta_2) \in \mathbb{T}_s$, we may define an associated function $f_{(\bar{\alpha},\bar{\beta})}$ on $\mathbb{M}_s$ as follows: 
\begin{equation}\label{eq: def f ab}
f_{(\bar{\alpha},\bar{\beta})}(w_1,w_2,z_1,z_2):=\alpha_1w_1 + \alpha_2 w_2 + \beta_1 z_1 + \beta_2 z_2 = \alpha_1 w_1 + \alpha_2 w_2 + \beta_1 z_1
\end{equation}
where the last equality holds since $\beta_2=0$ by assumption. Notice that $f_{(\bar{\alpha},\bar{\beta})}$ is simply the restriction to $\mathbb{M}_s$ of the usual standard inner product pairing with $(\bar{\alpha},\bar{\beta})$, but since $\mathbb{M}_s$ is not a linear subspace (additive subgroup) of $\Z^2 \times \Z^2$, we cannot discuss linearity on $\mathbb{M}_s$, and in particular, it is not a linear map. Using a sequence of bijections $\pi_1: \MM_s \to  M_1$, $\Theta_1: M_1 \to M_s(1)$, and $\Psi_1^{-1}: M_s(1) \to \mathbb{M}_s$, we may pullback the function $f_{(\bar{\alpha},\bar{\beta})}$ to $\MM_s$, thus defining a function on $\MM_s$. 
The following lemma shows that this association gives a bijection from $\mathbb{T}_s$ to $\Sp(\MM_s)$. 

\begin{Lemma}\label{lemma: Ts and Sp MMs}
Let $\Phi$ denote the association $(\bar{\alpha}, \bar{\beta}) \in \mathbb{T}_s \mapsto p_{(\bar{\alpha},\bar{\beta})} := f_{(\bar{\alpha},\bar{\beta})} \circ \Psi_1^{-1} \circ \Theta_1 \circ \pi_1: \MM_s \to \Z$ where $f_{(\bar{\alpha},\bar{\beta})}$ is the function on $\mathbb{M}_s$ defined above. Then 
\begin{enumerate} 
\item $\Phi$ has image $\Sp(\MM_s)$, and
\item $\Phi$ defines a bijection from $\mathbb{T}_s$ to $\Sp(\MM_s)$, and 
\item $\Phi$ respects addition when restricted to $\mathbb{T}_s(i)$ for $i=1$ or $i=2$. More precisely, for fixed $i$ and for $(\bar{\alpha},\bar{\beta}),(\bar{\alpha}', \bar{\beta}') \in \mathbb{T}_s(i)$, we have $\Phi((\bar{\alpha},\bar{\beta})+(\bar{\alpha}', \bar{\beta}')) = \Phi(\bar{\alpha},\bar{\beta}) + \Phi(\bar{\alpha}',\bar{\beta}')$ as functions on $\MM_s$. 
\end{enumerate} 
\end{Lemma}

\begin{proof} 
Since $\pi_1$ is a bijection, a point in $\Sp(\MM_s)$ can be uniquely described by its values on $M_1$, so we may consider instead the function $f_{(\bar{\alpha},\bar{\beta})} \circ \Psi_1^{-1} \circ \Theta_1: M_1 \to \Z$ and verify that it has the form given in Proposition~\ref{proposition: points of Sp MMs}. From the formulas for $\Theta_1$ and $\Psi_1^{-1}$ it is straightforward to compute that 
\begin{equation}\label{eq: pullback fab to M1}
f_{(\bar{\alpha},\bar{\beta})}\circ \Psi_1^{-1} \circ \Theta_1(x,y) = 
\begin{cases}
-\beta_1 x + (s \cdot \beta_1 -\alpha_2) y \, \, \, \textup{ if } y \leq 0, \\
    -\beta_1 x + \alpha_1 y \, \, \quad \quad \quad \quad \textup{ if } y \geq 0. \\
    \end{cases}
\end{equation}
In notation motivated by Proposition~\ref{proposition: points of Sp MMs}, set $p(\mathfrak{e}_1) = - \beta_1, p(\mathfrak{e}_2) = \alpha_1$, and $p(\mathfrak{e}'_2) = -(s\cdot \beta_1-\alpha_2)$. Now by Proposition
\ref{proposition: points of Sp MMs}, the function~\eqref{eq: pullback fab to M1} is an element of $\Sp(\MM_s)$ if and only if $p(\mathfrak{e}_2)+p(\mathfrak{e}'_2) = \min\{0,s \cdot p(\mathfrak{e}_1)\}$. We have 
\begin{equation*}
\begin{split}
    p(\mathfrak{e}_2)+p(\mathfrak{e}'_2) = \min\{0,s \cdot p(\mathfrak{e}_1)\} & \iff \alpha_1+\alpha_2-s\cdot \beta_1= \min\{0, -s \cdot \beta_1\}\\
    & \iff \alpha_1+\alpha_2 = \min\{s \cdot \beta_1, 0\} \\
    & \iff \alpha_1+\alpha_2 = s \min\{\beta_1, 0\} \, \, \textup{ since } s \geq 0. \\
\end{split}
\end{equation*}
From the above reasoning, it follows that if $(\alpha_1,\alpha_2,\beta_1,\beta_2=0)$ is in $\mathbb{T}_s$, then~\eqref{eq: pullback fab to M1} is in $\Sp(\MM_s)$. Moreover, for any $p \in \Sp(\MM_s)$ with corresponding values of $p(\mathfrak{e}_1), p(\mathfrak{e}_2)$ and $p(\mathfrak{e}'_2)$, we can take $\alpha_1 = p(\mathfrak{e}_2), \beta_1 = - p(\mathfrak{e}_1), \alpha_2 = p(\mathfrak{e}'_2) + s \cdot \beta_1$ and $\beta_2=0$ to obtain $p$ as $\Phi(\alpha_1,\alpha_2,\beta_1,\beta_2)=p$, so $\Phi$ is a bijection as claimed. 

The last claim follows from the fact that $f_{(\overline{\alpha},\overline{\beta})}$ is defined as the restriction of the standard inner product, which is linear in both variables, and the fact that $\mathbb{T}_s(i)$ is closed under addition for both $i=1$ and $i=2$. 
\end{proof}

Motivated by the above proof, we define the following bijection between $\mathcal{T}_s$ and $\mathbb{T}_s$: 
\begin{equation}\label{eq: def upsilon}
\Upsilon: \mathcal{T}_s \to \mathbb{T}_s, \quad (a,b,c) \mapsto (a,b-sc,-c,0).
\end{equation}
Next, we wish to translate the self-dual pairing $\mathbf{w}_s: \MM_s \to \Sp(\MM_s)$ from Section \ref{sec: MM_s} into the language of $\mathbb{M}_s$ and $\mathbb{T}_s$. Using the identifications $\Psi_1: \mathbb{M}_s \to M_s(1), \Theta^{-1}_1: M_s(1) \to M_1$, and $\Upsilon: \mathcal{T}_s \to \mathbb{T}_s$, it is straightforward to compute that when the map $\tilde{\mathbf{w}}_s := \Upsilon \circ \mathbf{w}_s \circ \Theta^{-1}_1 \circ \Psi_1: \mathbb{M}_s \to \mathbb{T}_s$ given by composition of $\mathbf{w}_s$ with these identifications it is given by 
\begin{equation}\label{eq: def tilde ws} 
\tilde{\mathbf{w}}_s(w_1, w_2, z_1, z_2) = 
\begin{cases} 
(z_2, -z_2 - s(w_1-w_2), -(w_1-w_2), 0) \quad \quad \textup{ if } \, w_1-w_2 \geq 0 \\
(z_2, -z_2, -(w_1-w_2), 0) \quad \quad \quad \quad \quad \quad \textup{ if } \, w_1-w_2 \leq 0. 
\end{cases} 
\end{equation}
Recall that tuples $(w_1, w_2,z_1,z_2)$ in $\mathbb{M}_s$ satisfy $\min\{w_1,w_2\}=0$. Thus the conditions $w_1-w_2 \geq 0$ and $w_1-w_2 \leq 0$ can be rephrased as $w_2=0$ and $w_1=0$ respectively, and this reformulation is also used below.

Let us now consider the following algebra:
\begin{equation}\label{eq: def algebra As}
\mathcal{A}_s = \mathbb{C}[x_1,x_2,y_1,y_2,y_1^{-1},y_2^{-1}]/\langle x_1x_2 - y_1^s -y_2^s, ~y_2 -1\rangle.
\end{equation}
It is straightforward to see that $\mathcal{A}_s$ is a Noetherian $\C$-algebra and an integral domain. 
(These are examples of cluster algebras of rank 1, see \cite[Section 3.2]{FominWilliamsZelevinsky}.)

Our main goal of this section is to prove that $\mathcal{A}_s$ can be equipped with a valuation $\fv_s$ in such a way that the pair $(\mathcal{A}_s, \fv_s)$ is a detropicalization of $\MM_s$. 
To do this, we first identify an additive basis of $\mathcal{A}_s$. Consider the following set: 
\begin{equation}\label{eq: def Bs basis}
\mathbb{B}_s := \{x_1^{w_1}x_2^{w_2}y_1^{z_1}y_2^{z_2} ~|~ (w_1,w_2,z_1,z_2) \in \mathbb{M}_s\} \subset \mathcal{A}_s.
\end{equation}
To see that $\mathbb{B}_s$ forms an additive basis for $\mathcal{A}_s$, we may argue in two steps. First suppose that the defining ideal consists of the single relation $x_1x_2 - y_1^s - y_2^s$. Then,  
since there is a monomial ordering $<$ such that the initial term of this relation is $x_1x_2$, it is immediate from standard results of Gr\"obner bases \cite[Proposition 1.1]{Sturmfels-Grobner} that the monomials $x_1^{w_1}x_2^{w_2}y_1^{z_1}y_2^{z_2}$ with $\min\{w_1,w_2\}=0$ form a basis for $\C[x_1,x_2,y_1,y_2]/\langle x_1x_2 - y_1^s - y_2^s\rangle$. For $\mathcal{A}_s$, however, we also have the additional defining relation $y_2=1$. This means that we may take as additive basis a set of monomials where the exponent $z_2$ of $y_2$ is uniquely determined by the exponents on the other variables. (It would be conventional simply to pick $z_2=0$ at all times, but it will be more convenient for us to pick $z_2$ to be a function of $w_1,w_2,z_1$.) In our setting, we choose $z_2 = -z_1-s w_2$. This argument shows that $\mathbb{B}_s$ is an additive basis of $\mathcal{A}_s$. We record this statement in the following. 

\begin{Lemma}
The image of the set $\mathbb{B}_s$ under the projection $\C[x_1,x_2,y_1^{\pm}, y_2^{\pm}] \to \mathcal{A}_s$ is an additive basis for $\mathcal{A}_s$. In particular, there is a one-to-one correspondence between $\mathbb{B}_s$ and $\mathbb{M}_s$, given by taking the exponent vector of a monomial in $\mathbb{B}_s$. 
\end{Lemma}

We are now ready to define a valuation $\mathfrak{v}_s$ which realizes $\mathcal{A}_s$ as a detropicalization of $\MM_s$. Let $\mathfrak{v}_s: \mathbb{B}_s \rightarrow \Sp(\MM_s)$ be the function defined as follows. For any $(w_1, w_2, z_1, z_2) \in \mathbb{M}_s$, we have just seen that the monomial $x_1^{w_1}x_2^{w_2}y_1^{z_1}y_2^{z_2}$ is in $\mathbb{B}_s$ and we define: 
\begin{equation}\label{eq: def vs on Bs} 
\mathfrak{v}_s(x_1^{w_1}x_2^{w_2}y_1^{z_1}y_2^{z_2}) := \Phi \circ \tilde{\mathbf{w}}_s(w_1,w_2,z_1,z_2)
\end{equation}
where $\Phi: \mathbb{T}_s \to \Sp(\MM_s)$ is the bijection constructed in Lemma~\ref{lemma: Ts and Sp MMs} and $\tilde{\mathbf{w}}_s$ is as in \eqref{eq: def tilde ws}. 
 
Since $\mathbb{B}_s$ is a basis of $\mathcal{A}_s$, we may then extend $\mathfrak{v}_s$ to a function on the algebra $\mathcal{A}_s$ by defining 
\begin{equation}\label{eq: def vs on As}
\mathfrak{v}_s(\sum \lambda_i \mathbb{b}_i) := \bigoplus \mathfrak{v}_s(\mathbb{b}_i) \in \P_{\MM_s} 
\end{equation}
for any linear combination $\sum_i \lambda_i \mathbb{b}_i$ of elements $\mathbb{b}_i \in \mathbb{B}_s$ where $\lambda_i \in \mathbb{C}$ are scalars. By definition we also set $\fv_s(0) := \infty$. 
Recall that we think of elements of $\Sp(\MM_s)$ as piecewise-linear functions on $\MM_s$ and the operation $\oplus$ is the min-combination of functions.

We are now ready to state and prove the main theorem. 

\begin{Theorem}\label{theorem: MMs detrop}
Let $s$ be a positive integer and let $\MM_s$ be the polyptych lattice associated to $s$ defined in Section~\ref{sec: MM_s}. Let $\mathcal{A}_s$ be the $\C$-algebra defined in~\eqref{eq: def algebra As} and let $\mathfrak{v}_s: \mathcal{A}_s \to \P_{\MM_s}$ be the map defined in~\eqref{eq: def vs on As}. Then: 
\begin{enumerate} 
\item $\mathfrak{v}_s: \mathcal{A}_s \to \P_{\MM_s}$ is a valuation on $\mathcal{A}_s$ with values in the idempotent semialgebra $\P_{\MM_s}$ in the sense of Definition~\ref{def_valuation}, and 
\item the pair $(\mathcal{A}_s, \mathfrak{v}_s)$ is a detropicalization of $\MM_s$ in the sense of Definition~\ref{definition: detropicalization}. 
\end{enumerate} 
\end{Theorem}

\begin{proof}[Proof of Theorem~\ref{theorem: MMs detrop}] 
We begin with the claim (1). To prove it, we must check the conditions for a valuation as listed in Definition~\ref{def_valuation}. Suppose that $f, g \in \mathcal{A}_s$. We wish to prove that $\fv_s(fg) = \fv_s(f) \odot \fv_s(g)$. Recalling that the $\odot$ operation in $\P_{\MM_s}$ is given by pointwise addition of functions, this is equivalent to showing that $\fv_s(fg) = \fv_s(f) + \fv_s(g)$ as functions on $\MM_s$. 

We take cases.  First suppose that $f=\bb=x_1^{w_1}x_2^{w_2}y_1^{z_1}y_2^{z_2} \in \mathbb{B}_s, g=\bb'=x_1^{w'_1}x_2^{w'_2}y_1^{z'_1}y_2^{z'_2} \in \mathbb{B}_s$, so both $f$ and $g$ are monomials with exponent vectors contained in $\mathbb{M}_s$, and additionally assume that the product monomial $\bb \bb'$ is in $\mathbb{B}_s$, i.e., $(w_1+w'_1,w_2+w'_2,z_1+z'_1,z_2+z'_2)$ is contained in $\mathbb{M}_s$. By definition, $\fv_s(\bb \bb') = \Phi \circ \tilde{\mathbf{w}}_s(w_1+w'_1, w_2+w'_2, z_1+z'_1,z_2+z'_2)$, and if $\bb,\bb' \in \mathbb{B}_s$ then this implies that either $w_1=w'_1=0$ or $w_2=w'_2=0$. In either case, the definition of $\tilde{\mathbf{w}}_s$ in~\eqref{eq: def tilde ws} implies that $\tilde{\mathbf{w}}_s(w_1+w'_1, w_2+w'_2,z_1+z'_1, z_2+z'_2) = \tilde{\mathbf{w}}_s(w_1,w_2,z_1,z_2) + \tilde{\mathbf{w}}_s(w'_1,w'_2,z'_1,z'_2)$, and moreover, all three images under $\tilde{\mathbf{w}}_s$ lie in $\mathbb{T}_s(i)$ for some $i$. Then Lemma~\ref{lemma: Ts and Sp MMs}(3) implies that $\Phi \circ \tilde{\mathbf{w}}_s$ is also additive on $(w_1,w_2,z_1,z_2)+(w'_1,w'_2,z'_1,z'_2)$, so by definition $\fv_s(\bb \bb')=\fv_s(\bb)+\fv_s(\bb')$ in this case. 

Next we consider the case $f=\bb=x_1^{w_1}x_2^{w_2}y_1^{z_1}y_2^{z_2} \in \mathbb{B}_s, g=\bb'=x_1^{w'_1}x_2^{w'_2}y_1^{z'_1}y_2^{z'_2} \in \mathbb{B}_s$, where this time we suppose that $\bb \bb'$ is not in $\mathbb{B}_s$. This means that $w_1+w'_1>0$ and $w_2 + w'_2 >0$. Since we know that $\min\{w_1,w_2\}=0=\min\{w'_1,w'_2\}$ by assumption, we may assume without loss of generality that $w_1=0,w_2>0, w'_1>0, w'_2=0$. In order to prove $\fv(\bb \bb')=\fv(\bb)+\fv(\bb')$, we compute both sides as functions on $\MM_s$. Since $\bb,\bb' \in \mathbb{B}_s$, the RHS may be computed from the definitions to be 
$$
\Phi\circ \tilde{\mathbf{w}}_s(w_1,w_2,z_1,z_2) + \Phi \circ \tilde{\mathbf{w}}_s(w'_1,w'_2,z'_1,z'_2) = f_{(z_2,-z_2,w_2,0)} + f_{(z'_2, - z'_2-sw'_1, -w'_1,0)}.
$$
(Here by slight abuse of notation we view functions on $\mathbb{M}_s$ as functions on $\MM_s$ via the identifications we established above.) For the LHS, we must first express $\bb \bb'$ as a linear combination of monomials in $\mathbb{B}_s$. Consider the case $w'_1 \leq w_2$. Then we have 
\begin{equation}
    \begin{split}
        x_1^{w'_1}x_2^{w_2} y_1^{z_1+z'_1}y_2^{z_2 + z'_2} & = (x_1x_2)^{w'_1} x_2^{w_2-w'_1} y_1^{z_1+z'_1}y_2^{z_2+z'_2} \\
        & = (y_1^s+y_2^s)^{w'_1} x_2^{w_2-w'_1} y_1^{z_1+z'_1} y_2^{z_2+z'_2} 
    \end{split}
\end{equation}
and the expansion of $(y_1^s+y_2^s)^{w'_1}$ contains monomials of the form $y_1^{sk} y_2^{s(w'_1-k)}$, for $0 \leq k \leq w'_1$, so we conclude that $x_1^{w'_1}x_2^{w_2} y_1^{z_1+z'_1}y_2^{z_2 + z'_2}$ can be expressed as a linear combination of the monomials 
$$
x_2^{w_2-w'_1} y_1^{z_1+z'_1 + sk} y_2^{z_2+z'_2 + s(w'_1-k)} \quad \quad \textup{ for } \, 0 \leq k \leq w'_1.
$$
We claim that the above monomials are in $\mathbb{B}_s$. Indeed, by assumption we have $z_1+z_2=-sw_2$ and $z'_1+z'_2=0$, so $z_1+z'_1+sk + z_2+z'_2+s(w'_1-k) = -sw_2+sw'_1$ as required. Thus by the definition of $\fv_s$ 
we compute $\fv_s(\bb \bb')$ by taking the minimum 
\begin{equation}
\begin{split} 
\fv_s(\bb \bb') & = \min\left\{ \fv_s(x_2^{w_2-w'_1} y_1^{z_1+z'_1 + sk} y_2^{z_2+z'_2 + s(w'_1-k)}) \, \mid \, 0 \leq k \leq w'_1 \right\} \\
& = \min\left\{ \Phi(z_2+z'_2 + s(w'_1-k), -z_2-z'_2-s(w'_1-k),w_2-w'_1, 0) \, \mid \, 0 \leq k \leq w'_1 \right\} \\
& = \min\left\{ f_{(z_2+z'_2+s(w'_1-k), -z_2-z'_2-s(w'_1-k), w_2-w'_1,0)} \, \mid \, 0 \leq k \leq w'_1 \right\} \\
& = \min\left\{ f_{(z_2, -z_2, w_2,0)} + f_{(z'_2, -z'_2-sw'_1, -w'_1, 0)} + g_{(s(w'_1-k), sk, 0,0)} \, \mid \, 0 \leq k \leq w'_1 \right\} \\
& = f_{(z_2, -z_2, w_2,0)} + f_{(z'_2, -z'_2-sw'_1, -w'_1, 0)} + \min\{ g_{(s(w'_1-k), sk, 0,0)} \, \mid \, 0 \leq k \leq w'_1 \}
\end{split}
\end{equation}
 where $g_{(c,d,0,0)}$ denotes the function on $\mathbb{M}_s$ defined by $g_{(c,d,0,0)}(w_1,w_2,z_1,z_2)=cw_1+dw_2$. For any element $(w_1,w_2,z_1,z_2) \in \mathbb{M}_s$, we have 
 $$
 \min\{ g_{(s(w'_1-k), sk, 0,0)}(w_1,w_2,z_1,z_2) \, \mid \, 0 \leq k \leq w'_1 \} = \min\{ s(w'_1-k)w_1+skw_2 \, \mid \, 0 \leq k\leq w'_1 \} = 0 
 $$
because $\min\{w_1,w_2\}=0$ for an element in $\mathbb{M}_s$. Hence the function $\min\{ g_{(s(w'_1-k), sk, 0,0)} \, \mid \, 0 \leq k \leq w'_1 \}$ is identically $0$ on $\mathbb{M}_s$, and we conclude 
$$
\fv_s(\bb \bb') = f_{(z_2, -z_2, w_2,0)} + f_{(z'_2, -z'_2-sw'_1, -w'_1, 0)} = \fv_s(\bb)+\fv_s(\bb')
$$
as desired. The case $w_2<w'_1$ follows similarly. This proves that $\fv_s(\bb \bb')=\fv_s(\bb)+\fv_s(\bb')$ for any two elements $\bb,\bb' \in \mathbb{B}_s$.

We must next prove that $\fv_s(fg)=\fv_s(f) + \fv_s(g)$ for arbitrary $f,g \in \mathcal{A}_s$. However, the argument is the same as that given in \cite[Lemma 8.10]{EscobarHaradaManon-PL} so we do not reproduce it here. 
The remaining properties of valuations in Definition~\ref{def_valuation} are straightforward to verify and are left to the reader. 

We now prove that the pair $(\mathcal{A}_s, \fv_s: \mathcal{A}_s \to \P_{\MM_s})$ is a detropicalization of $\MM_s$ in the sense of Definition~\ref{definition: detropicalization}. We have already shown that $\fv_s: \mathcal{A}_s \to \P_{\MM_s}$ is a valuation (with values in the idempotent semialgebra $\P_{\MM_s}$), so it remains only to show that every element of $\Sp(\MM_s)$ is in the image of $\fv_s$, i.e. that $\fv_s$ is surjective onto $\Sp(\MM_s)$, and that the Krull dimension of $\mathcal{A}_s$ is equal to the rank of $\MM_s$. 
The first claim follows immediately from the fact that $\fv_s$ restricted to $\mathbb{B}_s$ is a bijection from $\mathbb{B}_s$ to $\Sp(\MM_s)$, as was seen above. The second claim follows from the fact that $\Spec(\mathcal{A}_s)$ is an affine variety of dimension $2$, so the Krull dimension of $\mathcal{A}_s$ is $2$, which is the rank of $\MM_s$, as required. 
\end{proof}

\section{Example: a Cox ring of a compactification $X_{\mathcal{A}_\MM}(\PP)$}\label{sec: cox ring}

In \cite[Section 7]{EscobarHaradaManon-PL}, the authors establish a general framework for constructing a compactification of the variety $\Spec(\mathcal{A}_{\MM})$ (where $\mathcal{A}_{\MM}$ is a detropicalization of a polyptych lattice $\MM$) with respect to a choice of PL polytope $\PP \subset \MM_{\R}$.  Moreover, in the case when $\mathcal{A}_\MM$ is a UFD, it is shown that the Cox ring of the compactification is finitely generated. The main purpose of this section is to illustrate the general theory outlined in \cite{EscobarHaradaManon-PL} by working out, in detail, the Cox ring of the compactification of $\Spec(\mathcal{A}_{s})$ with respect to a PL polytope $\PP$.  More specifically, we showed in \cite[Theorem 7.19]{EscobarHaradaManon-PL} that both the class group and the Cox ring of the compactification is finitely generated. Here, for the rank-$2$ example $\MM_s$ for $s=1$ and for a specific PL polytope $\PP$, we take a step further: we give a concrete presentation of the Cox ring in terms of generators and relations. 
From the presentation given in Proposition~\ref{prop: Cox ring for rank 2 ex} it follows that $X_{\Aa_s}(\PP)$ is not a toric variety. We also note that recent work of Oda \cite[Theorem 8.3]{Oda2025} gives a computation of Cox rings in more generality, for a broad class of PL polytopes in $\MM_2$.

Let $s=1$. We note first that it follows from the explicit generators-and-relation presentation of $\Aa_s$ for the case $s=1$ that $\Aa_s$ is a UFD, since the localization of a UFD at a prime element is still a UFD.  Therefore, \cite[Theorem 7.19]{EscobarHaradaManon-PL} applies. Next, we specify the PL polytope $\PP$ in question. As in Section~\ref{sec: GF polytope}, we specify points under the identification $\Sp(\MM_s) \cong \mathcal{T}_s$. With this understanding, 
we consider the following $4$ points 
$$
\mathsf{p}_1= (-1,1,1), 
\quad
\mathsf{p}_2 = (1,-1,1), 
\quad
\mathsf{p}_3 = (1,-2,-1),
\quad 
\mathsf{p}_4 = (-2,1,-1)
$$
and define 
\begin{equation}\label{eq: PP for Cox}
\PP := \mathcal{H}_{\mathsf{p}_1,-1} \cap \mathcal{H}_{\mathsf{p}_2,-1} \cap \mathcal{H}_{\mathsf{p}_3,-1} \cap 
\mathcal{H}_{\mathsf{p}_4,-1}. 
\end{equation}
By Figure~\ref{fig_cox_computation_polytope} below we see that this is compact, and hence a PL polytope. 
 Note that $\PP$ is  not an integral PL polytope. 
 \begin{figure}[h]
    \centering
    \begin{tikzpicture}
    \filldraw[fill=gray!50] (1,0)--(-1/3,2/3)--(-1,0)--(-1/3,-2/3)--(1,0);
    \draw[gray,<->] (-1.5,0)--(1.5,0);
    \draw[gray,<->] (0,-1)--(0,1);
    \foreach \i in {-1,...,1}
      \foreach \j in {-1,...,1}{
        \filldraw[black] (\i,\j) circle(.5pt);
      };
    \end{tikzpicture}
    \hspace{5cm}
    \begin{tikzpicture}
    \filldraw[fill=gray!50] (1,0)--(1/3,2/3)--(-1,0)--(-1/3,-2/3)--(1,0);
    \draw[gray,<->] (-1.5,0)--(1.5,0);
    \draw[gray,<->] (0,-1)--(0,1);
    \foreach \i in {-1,...,1}
      \foreach \j in {-1,...,1}{
        \filldraw[black] (\i,\j) circle(.5pt);
      };
    \end{tikzpicture}
    \caption{The two chart images of the PL polytope $\PP$ in $\MM_1 \otimes \R$. The image in $M_1$ is on the left and $M_2$ on the right.}
    \label{fig_cox_computation_polytope}
\end{figure}

We now briefly recall the definition of the compactification. For details we refer to \cite[Section 7]{EscobarHaradaManon-PL}. For $k$ a positive integer, we define the polytope $k\PP$ by scaling the parameters in the defining inequalities, so in our case
$$
k \PP := 
\mathcal{H}_{\mathsf{p}_1,-k} \cap \mathcal{H}_{\mathsf{p}_2,-k} \cap \mathcal{H}_{\mathsf{p}_3,-k} \cap 
\mathcal{H}_{\mathsf{p}_4,-k}.
$$
We briefly remark that $3\PP$ is an integral PL polytope and the associated Cox ring is the same as that of $\PP$. 
We also define 
\begin{equation}\label{eq_gamma_kdelta}
    \Gamma(\mathcal{A}_s, k\PP) := \{f \in \mathcal{A}_s \mid \fv(f) \geq \psi_{k\PP}\}
\end{equation}
where $\psi_{k\PP}: \NN=\MM_s \to \mathbb{R}$ denotes the support function of the PL polytope $k\PP$
and the inequality is with respect to the partial order on $\P_\NN = \P_{\MM_s}$ (i.e. pointwise inequality of functions). 
Recall also that the \textbf{support of $f \in \mathcal{A}_s$} is defined as 
follows. If $f=\sum\lambda_ib_i$ for $\lambda_i \in \K$ and $b_i \in \mathbb{B}_s$ is an element in $\mathcal{A}_s$ expressed uniquely as a linear combination of elements of $\mathbb{B}_s$, the \textbf{support of $f$}, denoted by $\mathrm{supp}(f)$, is the point-convex hull of $\{\v^{-1}(\fv(b_i))\mid \lambda_i\neq 0\}$ in $\MM_s \otimes \Q$.

It is shown in \cite[Lemma 7.4]{EscobarHaradaManon-PL} that the space $\Gamma(\mathcal{A}_s, k\PP)$ can be equivalently described as 
\begin{equation}\label{eq: level k support in kP}
\Gamma(\mathcal{A}_s, k\PP) = \{ f \in \mathcal{A}_s \, \mid \, \mathrm{supp}(f) \subseteq k\PP \}.
\end{equation}
We will use this characterization. Then the PL polytope algebra $\Aa_s^\PP$ is defined as 
\begin{equation}
    \Aa_s^\PP := \bigoplus_{k \geq 0} \Gamma(\Aa_s, k\PP) \cdot t^k = \bigoplus_{k \geq 0} \{f \in \Aa_s \, \mid \,  \mathrm{supp}(f) \subseteq k\PP\} \cdot t^k
\end{equation}
where the last equality is by~\eqref{eq: level k support in kP}.  The algebra $\Aa_s^\PP$ is evidently $\Z_{\geq 0}$-graded by the degree of $t$, and 
we define the \textbf{compactification of $\Spec(\Aa_s)$ with respect to $\PP$} as
\begin{equation*}\label{eq: def compactification wrt P}
X_{\Aa_s}(\PP) := \Proj(\Aa_s^\PP).
\end{equation*}

In preparation for the computation of the Cox ring of $X_{\Aa_s}(\PP)$, it will be useful to prove some general results. The utility of these results in relation to the Cox ring computation will become apparent below when we explain the general method of computation, which is derived from \cite[Construction 1.4.2.1]{CoxRingsBible}. 
We emphasize that Lemma~\ref{lemma: units} and Proposition~\ref{prop: units} apply to any finite polyptych lattice $\MM$ over $\Z$, not just the rank-$2$ $\MM_s$ case. Recall that a detropicalization $\Aa_\MM$ is a Noetherian $\K$-algebra which is an integral domain, equipped with a valuation $\fv: \Aa_\MM \to S_\MM$. Following \cite{EscobarHaradaManon-PL}, we assume that $\K$ is algebraically closed.

\begin{Lemma}\label{lemma: units}
Let $\MM$ be a finite polyptych lattice of rank $r$ over $\Z$ with a fixed choice of strict dual ($\Z$-)pair $(\MM, \NN, \v, \w)$. Let $(\Aa_\MM, \fv)$ be a detropicalization of $\MM$ with convex adapted basis $\mathbb{B} = \{\bb_m\}_{m \in \MM}$ (with $\fv(\bb_m) = m \in \MM \subset S_\MM \cong \P_\NN$). Suppose $f,g \in \Aa_\MM$ and $f \cdot g = \bb_{m_0}$ for some $m_0 \in \MM$. 
Let $\beta \in \pi(\NN)$ such that $m_0 \in C_\beta := \v^{-1}(\Sp(\NN, \beta))$. Then $\fv(f)=m, \fv(g)=m'$ for some $m, m'\in C_\beta \cap \MM$, and $f = c \, \bb_{m}, g = c' \, \bb_{m'}$ for some $c, c' \neq 0, c,c' \in \mathbb{K}$. 
\end{Lemma}

\begin{proof} 
If $f \cdot g = \bb_{m_0}$ then $\fv(f)+\fv(g)=\fv(\bb_{m_0})=m_0$. Since $m_0$ is assumed to be in $C_\beta = \v^{-1}(\Sp(\NN,\beta))$, this implies that, interpreted as a function on $\NN$ via $\v$, $m_0$ induces a linear function on the coordinate chart $N_\beta$. By definition of $\fv$, the images $\fv(f),\fv(g)$ are convex piecewise-linear functions on $N_\beta$, and their sum is linear on $N_\beta$. This can occur only if both $\fv(f)$ and $\fv(g)$ are linear on $N_\beta$. Thus $\fv(f),\fv(g)$ are contained in $\v^{-1}(\Sp(\NN,\beta)) = C_\beta \cap \MM$. In particular there exist $m, m' \in C_\beta \cap \MM$ such that $\fv(f)=m, \fv(g)=m'$. Since $\mathbb{B}$ is a convex adapted basis for $\fv$, it now follows that $f = c \bb_{m} + \sum_{i \in \mathcal{I}} c_i \bb_{m_i}$ and $g = c' \bb_{m'} + \sum_{j \in \mathcal{J}} c_j \bb_{m_j}$ for $c,c', c_i, c_j \in \mathbb{K}^*$. Moreover, since $\fv(f)=\min(\{m\} \cup \{m_i\}_{i\in \mathcal{I}}) = m$, we must have that $\fv(\bb_{m_i}) = m_i \geq m$ as functions on $\NN$ for all $i$, and similarly, $m_j \geq m'$ as functions on $\NN$ for all $j$. But $m_i,m_j$ are convex functions on $N_\beta$ obtained as a minimum of a finite set of linear functions, so the relations $m_i \geq m, m_j \geq m'$ as functions on $N_\beta$ are only possible if $m_i=m, m_j=m'$. In other words, $f = c \bb_m$ and $g=c'\bb_{m'}$. This concludes the proof. 
\end{proof} 

We can now compute the group of units in $\Aa_\MM$.

\begin{Proposition}\label{prop: units}
Let the notation and assumptions be as in Lemma~\ref{lemma: units}. Let $u \in \Aa_\MM$. Then $u$ is a unit in $\Aa_\MM$ if and only if $u = c \,\bb_{m}$ for $c \in \mathbb{K}^*$ and $m \in \left(\bigcap_{C_\beta \in \Sigma(\MM)_{max}} C_\beta\right) \cap \MM$ where the intersection is over all maximal-dimensional cones in $\Sigma(\MM)$. 
\end{Proposition}

\begin{proof} 
Suppose first that $u$ is a unit. Then there exists $v \in \Aa_\MM$ with $u \cdot v =1$, and $1 = \bb_{0}$ for $0\in \MM$. Note also that $0 \in C_\beta$ for all maximal-dimensional cones $C_\beta$ in $\Sigma(\MM)$, since $C_\beta$ is a cone. By Lemma~\ref{lemma: units}, it follows that $v = c \, \bb_m$ for some $m \in \left( \bigcap_{C_\beta \in \Sigma(\MM)_{max}} C_\beta\right) \cap \MM$ and $c \in \mathbb{K}^*$. 

 Now for the opposite implication, suppose that $m \in \left(\bigcap_{C_\beta \in \Sigma(\MM)} C_\beta\right) \cap \MM$. The PL fan $\Sigma(\MM)$ is a complete fan, so the intersection of all its cones must be a linear subspace (in particular, addition is well-defined in this intersection), and thus $\left(\bigcap_{C_\beta \in \Sigma(\MM)} C_\beta \right) \cap \MM$ is a lattice. 
 Hence for any $m \in\left(\bigcap_{C_\beta \in \Sigma(\MM)} C_\beta \right) \cap \MM $ there exists $m' \in \left(\bigcap_{C_\beta \in \Sigma(\MM)} C_\beta \right) \cap \MM$ with $m+m'=0$. Now consider $\bb_m \cdot \bb_{m'}$, for which we may compute by properties of valuations that $\fv(\bb_m \cdot \bb_{m'})= m + m' = 0 = \fv(1)$. Then $\bb_m \cdot \bb_{m'} = c' \, 1 + \sum_i c_i \bb_{m_i}$ for $c' \in \mathbb{K}^*$ and $m_i \geq 0$ as functions on $\NN$ for all $i$. The same argument as in the proof of Lemma~\ref{lemma: units} shows $m_i \equiv 0$ on $\NN$, i.e. $m_i=0$, and so $\bb_m \cdot \bb_{m'}=c' \, 1$. An inverse of $c \cdot \bb_m$ is therefore given by $\frac{1}{c\, c'} \bb_{m'}$ and the claim is proved. 
\end{proof}

Turning back to our concrete rank-$2$ example $\MM_s$ for $s=1$, an application of Proposition~\ref{prop: units} immediately yields the following. 

\begin{Corollary}\label{corollary: units for s=1}
    Let $s=1$ and consider $\MM_s, \Aa_s$ as above. The group of units of $\Aa_s$ is generated by $y_1y_2^{-1}$. 
\end{Corollary}

\begin{proof} 
By Proposition~\ref{prop: units} we must find $m$ in the intersection of the maximal cones of $\Sigma(\MM_s)$. The (lattice points inside the) intersection of the two maximal cones in $\Sigma(\MM_s)$, viewed as a subset of $\mathbb{M}_s$, is the set $\{w_1=w_2=0\} \subset \mathbb{M}_s$. In this subset we must have $z_1+z_2=0$, so $z_2=-z_1$ and we see that the monomials -- i.e. the convex adapted basis elements -- corresponding to these lattice points are exactly $y_1^{k}y_2^{-k}$ for $k \in \Z$. These are generated as a group by the single generator $y_1 y_2^{-1}$. 
\end{proof}

We now explain our general method for computing the Cox ring of $X_{\Aa_\MM}(\PP)$, which is also applicable in general, not just our specific rank $2$ examples. It is based on \cite[Construction 1.4.2.1]{CoxRingsBible} as well as the proof of \cite[Theorem 7.16]{EscobarHaradaManon-PL}. First, in \cite{CoxRingsBible} it is explained that the Cox ring of $X$ (for $X$ an irreducible, normal prevariety with $\Gamma(X, \mathcal{O}^*)=\mathbb{K}^*$ and finitely generated class group) can be described as 
\begin{equation}\label{eq: Cox ring bible description} 
\Gamma(X, \mathcal{S})/\Gamma(X,\mathcal{I})
\end{equation} 
where $\mathcal{S}$ is a certain sheaf of divisorial algebras (and $\Gamma(X,\mathcal{S})$ the ring of its global sections) and $\mathcal{I}$ is a sheaf of ideals of $\mathcal{S}$. In this exposition, we do not give a detailed description of either $\mathcal{S}$ or $\mathcal{I}$ because we are able to give another, more concrete, description of both $\Gamma(X,\mathcal{S})$ and $\Gamma(X, \mathcal{I})$. Indeed, in the course of the proof of \cite[Theorem 7.16]{EscobarHaradaManon-PL} we show the following. 

\begin{Lemma}\label{lemma: Cox ring Aa_MM description}
Let $\MM$ be a finite polyptych lattice of rank $r$ over $\Z$ with a fixed choice of strict dual $\Z$-pair $(\MM,\NN, \v, \w)$. Let $\K$ be an algebraically closed field and $(\Aa_\MM,\fv)$ a detropicalization of $\MM$ with convex adapted basis $\B$. 
Let $\PP = \cap_{i=1}^{\ell} \HH_{\w(n_i), a_i} \subset \MM_\R$ be a full-dimensional PL polytope. Suppose $n_i \in \NN, n_i \neq 0$, the $n_i$ are pairwise distinct, and $a_i \in \Z_{< 0}$ for all $i \in [\ell]$. Suppose also that for each $n_i$ there exists a coordinate chart $\alpha_i \in \pi(\MM)$ on which $n_i$ is linear, and, the intersection of the boundary of $\pi_{\alpha_i}(\HH_{\w(n_i),a_i})$ with $\pi_{\alpha_i}(\PP)$ is a facet of $\pi_{\alpha_i}(\PP)$. Let $X_{\Aa_\MM}(\PP) := \mathrm{Proj}(\Aa^\PP_\MM)$ be the compactification of $\Spec(\Aa_\MM)$ constructed in \cite[Section 7]{EscobarHaradaManon-PL}. 
Then, for $X = X_{\Aa_\MM}(\PP)$, the ring $\Gamma(X,\mathcal{S})$ can be described as 
\begin{equation}\label{eq: Cox ring Aa_MM description}
    \bigoplus_{\overline{r}\in \Z^\ell} \Aa_\MM(\overline{r}) t_1^{r_1} \cdots t_\ell^{r_\ell} \subset \Aa_\MM[t_1^{\pm}, \cdots, t_\ell^{\pm}]
\end{equation}
where for $\overline{r} =(r_1,\cdots,r_\ell) \in \Z^\ell$ we define 
\begin{equation}\label{eq: def AaMMr}
\Aa_\MM(\overline{r}) := \mathrm{span}_{\mathbb{K}}\{ \bb_m \, \mid \, \langle n_i, m \rangle + r_i \geq 0 \, \, \textup{ for all } i \in [\ell] \}
\end{equation}
and $\langle n_i, m\rangle$ denotes the dual pairing between $\NN$ and $\MM$. 
\end{Lemma}

Continuing this line of reasoning, we can also describe $\Gamma(X,\mathcal{I})$ more concretely in terms of $\Aa_\MM$ as follows. For $f \in \Aa_\MM$, let $\overline{d}_f \in \Z^\ell$  denote the integer vector $\overline{d}_f := (\mathrm{ord}_{D_1}(f), \cdots, \mathrm{ord}_{D_\ell}(f))$ given by the orders of vanishing of $f$ along the divisors $D_i$ corresponding to the PL half-spaces $\HH_{\w(n_i),a_i}$ (as described in \cite[Section 7]{EscobarHaradaManon-PL}). Then we have the corresponding monomial $t^{\overline{d}_f} := t_1^{\mathrm{ord}_{D_1}(f)} t_2^{\mathrm{ord}_{D_2}(f)} \cdots t_\ell^{\mathrm{ord}_{D_\ell}(f)}$. 

\begin{Lemma}\label{lemma: kernel}
    Let the assumptions and notation be as in Lemma~\ref{lemma: Cox ring Aa_MM description}. Under the identification of $\Gamma(X,\mathcal{S})$ with~\eqref{eq: Cox ring Aa_MM description} given in Lemma~\ref{lemma: Cox ring Aa_MM description}, the ideal $\Gamma(X,\mathcal{I})$ in $\Gamma(X,\mathcal{S})$ is contained in the ideal in~\eqref{eq: Cox ring Aa_MM description} generated by $u - t^{\overline{d}_u}$, as $u$ ranges over the group of units in $\Aa_\MM$. 
\end{Lemma}

\begin{proof} 
We interpret the objects 
cited in \cite[Construction 1.4.2.1]{CoxRingsBible} in the same way as in \cite[Proof of Theorem 7.14]{EscobarHaradaManon-PL} so we will only sketch the argument. First, it is explained in \cite{CoxRingsBible} that $\Gamma(X,\mathcal{I})$ is generated by sections of the form $1-\chi(E)$ where $1$ is homogeneous of degree $0$, $E$ ranges over elements of the kernel $K^0$ of the surjection $\oplus_i \Z\cdot D_i \to \mathrm{Cl}(X)$ onto the class group of $X$, $\chi: K^0 \to \mathbb{K}(X)^*$ is a character, and $\chi(E)$ is homogeneous of degree $-E$. Since $K^0$ consists of divisors $E = \sum_i a_i D_i$ such that there exists a rational function $f \in \mathbb{K}(X)^*$ with $\mathrm{div}(f)=E$, and since the $D_i$ form the boundary of the compactification $X_{\Aa_\MM}(\PP)$ of $\Spec(\Aa_\MM)$, it follows that $K^0$ consists of $\sum_i a_i D_i$ for which there exists a unit $u \in \Aa_\MM$ with $\mathrm{div}(u)=\sum_i a_i D_i$. In particular, in the notation of \cite{CoxRingsBible}, $\chi(\sum_i a_i D_i)=u$. Moreover, our conventions on the homogeneous degree (encoded by the $t_i$ variables) imply that the relation $1-\chi(E)$ is equivalent to $u - t^{\overline{d}_u} = 0$. This proves the claim. 
\end{proof}

With these results in place we can now explain our method of computation. Let $\mathcal{J}$ denote the ideal in $\Aa_\MM[t_1^{\pm}, \cdots, t_\ell^{\pm}]$ generated by the elements $u - t^{\overline{d}_u} = 0$ as $u$ ranges over the group of units of $\Aa_\MM$. Then it follows from \cite[Construction 1.4.2.1]{CoxRingsBible},  Lemma~\ref{lemma: Cox ring Aa_MM description}, and Lemma~\ref{lemma: kernel} that there is an injective ring homomorphism 
$$
\Gamma(X,\mathcal{S})/\Gamma(X,\mathcal{I}) \hookrightarrow \Aa_\MM[t_1^{\pm}, \cdots, t_\ell^{\pm}]/\mathcal{J}.
$$
 The map is induced by the natural inclusion of~\eqref{eq: Cox ring Aa_MM description} into $\Aa_\MM[t_1^{\pm},\cdots, t_\ell^{\pm}]$. Thus, in order to give an explicit presentation of $\mathcal{R}(X)$, it suffices to determine a finite list of generators of~\eqref{eq: Cox ring Aa_MM description}, which we denote as $\{X_1,\cdots, X_n\}$, allowing us to  define the surjective ring homomorphism 
$$
\varphi: \C[u_1, u_2,\cdots, u_n] \rightarrow \bigoplus_{\overline{r}\in \Z^\ell} \Aa_\MM(\overline{r}) t_1^{r_1} \cdots t_\ell^{r_\ell}, \quad u_i \mapsto X_i. 
$$
Composing $\varphi$ with the inclusion 
$$
\bigoplus_{\overline{r}\in \Z^\ell} \Aa_\MM(\overline{r}) t_1^{r_1} \cdots t_\ell^{r_\ell} \hookrightarrow \Aa_\MM[t_1^{\pm},\cdots,t_\ell^{\pm}]
$$
and the quotient map
$$
\Aa_\MM[t_1^{\pm},\cdots,t_\ell^{\pm}] \to \Aa_\MM[t_1^{\pm},\cdots,t_\ell^{\pm}]/\mathcal{J}
$$
then gives a surjective map from $\C[u_1,\cdots,u_n]$ to a ring isomorphic to $\Gamma(X,\mathcal{S})/\Gamma(X,\mathcal{I})$. Computing the kernel $\kappa$ of this map then gives the desired presentation, namely
$$
\Gamma(X,\mathcal{S})/\Gamma(X,\mathcal{I}) \cong \C[u_1,\cdots,u_n]/\kappa.
$$

We now implement the above strategy in our case of $\MM_s$ for $s=1$, the detropicalization $\Aa_s$ of Section~\ref{sec: detrop MMs}, and the PL polytope $\PP$ of~\eqref{eq: PP for Cox}. In Corollary~\ref{corollary: units for s=1} we already computed the generator of the group of units to be $y_1y_2^{-1}$. By mapping the corresponding element of $\mathbb{M}_s$ to $M_1$ and then evaluating the point on the image, it is straightforward from the explicit description of the 
four points $\mathsf{p}_1=(-1,1,1), \mathsf{p}_2=(1,-1,1), \mathsf{p}_3=(1,-2,-1),\mathsf{p}_4=(-2,1,-1)$ (thought of as elements of $\mathcal{T}_s$) that the order of vanishing of $y_1y_2^{-1}$ along their corresponding divisors are given by 
$-1,-1,1,1$ respectively. Thus in our example we have 
$$
\mathcal{J} = \langle y_1y_2^{-1} - t_1^{-1}t_2^{-1} t_3 t_4 \rangle.
$$

Following the general method outlined above, our next step is to find a set of generators for~\eqref{eq: Cox ring Aa_MM description}. To accomplish this, we take the following approach. The description~\eqref{eq: def AaMMr} of the $\overline{r}$-graded piece of~\eqref{eq: Cox ring Aa_MM description} makes it clear that~\eqref{eq: Cox ring Aa_MM description} is spanned by $\bb_m t^{\overline{r}}$ where $m \in \MM$ and $\overline{r} \in \Z^\ell$ satisfy certain inequalities. We have seen in Section~\ref{sec: MM_s} that the PL fan $\Sigma(\MM_s)$ of $\MM_s$ has two maximal cones, $H_+$ and $H_-$, each of which is a half-space. Using the identification $\MM_s \cong \mathbb{M}_s$, (the lattices within) the half-space $H_-$ may be identified with 
$$
\mathbb{M}_s(1) := \{(0, w_2, z_1, z_2) \, \mid \,  w_2 \geq 0, - w_2 = z_1+z_2\}
$$
(here we have used $s=1$)
and the half-space $H_+$ with 
$$
\mathbb{M}_s(2) := \{(w_1, 0, z_1, z_2) \, \mid \, w_1 \geq 0, 0 = z_1 + z_2\}.
$$
Note that in both cases, $z_2$ is completely determined by the other variables, so $\mathbb{M}_s(i) \cong \Z_{\geq 0} \times \Z$ for both $i=1,2$. 

Now we consider the cases $i=1,2$ separately. First suppose $i=1$. Then for each $m = (0,w_2,z_1,z_2) \in \mathbb{M}_s(1)$ (here we are implicitly using the identification $\MM_s \cong \mathbb{M}_s$) the corresponding convex adapted basis element $\bb_m$ is $x_2^{w_2}y_1^{z_1}y_2^{z_2} = x_2^{w_2}y_1^{z_1}y_2^{-w_2-z_1}$. Here we have used the conditions for a vector to be in $\mathbb{M}_s(1)$ and have also used that $s=1$.  Note that $\mathbb{M}_s(1)$ is closed under addition. Now it follows from~\eqref{eq: def AaMMr} that the monomials $\bb_m t^{\overline{r}}$, with $m =(0,w_2,z_1,z_2) \in \mathbb{M}_s(1)$, lying in~\eqref{eq: Cox ring Aa_MM description} are precisely those satisfying 
$$
\mathsf{p}_1(0,w_2,z_1,z_2) + r_1 \geq 0, \quad
\mathsf{p}_2(0,w_2,z_1,z_2) + r_2 \geq 0, \quad
\mathsf{p}_3(0,w_2,z_1,z_2) + r_3 \geq 0,
\quad
\mathsf{p}_4(0,w_2,z_1,z_2) + r_4 \geq 0.
$$
A computation using Proposition~\ref{proposition: points of Sp MMs} shows that 
the basis elements of~\eqref{eq: Cox ring Aa_MM description} corresponding to $m \in \mathbb{M}_s(1)$ are in bijective correspondence with 
$$
T_1 = \{(z_2,w_2,r_1,r_2,r_3,r_4) \, \mid \, 
w_2 \geq 0, z_2+w_2+r_1 \geq 0,
z_2-w_2+r_2 \geq 0,
-z_2-2w_2+r_3 \geq 0,
-z_2+w_2+r_4 \geq 0\}.
$$

A similar computation shows that for $i=2$ there is a bijection between 
$$
T_2 = \{(z_2, w_1,r_1,r_2,r_3,r_4) \, \mid \, w_1 \geq 0, 
z_2-w_1+r_1 \geq 0,
z_2+w_1+r_2 \geq 0,
-z_2+ w_1+r_3 \geq 0,
-z_2-2w_1+r_4 \geq 0\}.
$$
and another subset of basis elements of~\eqref{eq: Cox ring Aa_MM description}, corresponding to $\mathbb{M}_s(2)$. 
Together, the union of these basis elements for $i=1$ and $i=2$ span all of~\eqref{eq: Cox ring Aa_MM description}. 
Notice that since $\mathbb{M}_s(i)$ is closed under addition for $i=1,2$, the corresponding monomials are closed under multiplication. Thus, if we find generators for the affine semigroups $T_1$ and $T_2$, then the union of the corresponding monomials will generate~\eqref{eq: Cox ring Aa_MM description}. 
A Macaulay2 computation reveals that $T_1$ is generated by 
$$
\left\{ 
\begin{pmatrix} 0 \\ 0 \\ 1 \\ 0 \\ 0 \\ 0  \end{pmatrix}, 
\begin{pmatrix} 0 \\ 0 \\ 0\\ 1 \\ 0 \\ 0 
\end{pmatrix}, 
\begin{pmatrix} 0 \\ 0 \\ 0 \\ 0 \\ 1 \\ 0 
\end{pmatrix},
\begin{pmatrix} 0 \\ 0 \\ 0 \\ 0 \\ 0 \\ 1 
\end{pmatrix}, 
\begin{pmatrix} 
1 \\ 0 \\ -1 \\-1 \\ 1 \\ 1 
\end{pmatrix}, 
\begin{pmatrix} 
-1 \\ 0 \\ 1 \\ 1 \\ -1 \\ -1 
\end{pmatrix} ,
\begin{pmatrix} 0 \\ 1 \\ -1\\ 1 \\ 2 \\ -1  \end{pmatrix}
\right\} 
$$
where the vector entries correspond to the variables $z_2, w_2,r_1,r_2,r_3,r_4$ respectively, 
while $T_2$ is generated by 
$$
\left\{ 
\begin{pmatrix} 
0 \\ 0  \\ 1 \\ 0 \\ 0 \\ 0 
\end{pmatrix},
\begin{pmatrix} 0 \\ 0 \\ 0 \\ 1 \\ 0 \\ 0  
\end{pmatrix}, 
\begin{pmatrix} 0 \\ 0 \\ 0 \\ 0 \\ 1 \\ 0 
\end{pmatrix},
\begin{pmatrix} 0 \\ 0 \\ 0 \\ 0 \\ 0 \\ 1 
\end{pmatrix}, 
\begin{pmatrix} 
1 \\ 0  \\ -1 \\ -1 \\ 1 \\ 1 
\end{pmatrix}, 
\begin{pmatrix} 
-1 \\ 0  \\ 1 \\ 1 \\ -1 \\ -1  
\end{pmatrix},
\begin{pmatrix} 0 \\ 1 \\ 1 \\ -1 \\ -1 \\ 2  \end{pmatrix}
\right\}
$$
where the vector entries correspond to $z_2, w_1 ,r_1,r_2,r_3,r_4$ respectively. 
It follows that the following eight monomials generate~\eqref{eq: Cox ring Aa_MM description}: 
$$
t_1, t_2, t_3,t_4, 
y_1^{-1}y_2 t_1^{-1}t_2^{-1}t_3 t_4 , 
y_1 y_2^{-1} t_1t_2 t_3^{-1} t_4^{-1}, 
x_1 t_1t_2^{-1} t_3^{-1} t_4^2, 
x_2 y_1^{-1} t_1^{-1}t_2 t_3^2 t_4^{-1} . 
$$
Temporarily labelling the above monomials as $X_1, \cdots, X_8$ from left to right, this implies that the ring homomorphism defined as 
$$
\tilde{G}: \C[u_1, \cdots, u_8] \rightarrow \Aa_\MM[t_1^{\pm}, t_2^{\pm}, t_3^{\pm},t_4^{\pm}], \quad \quad u_i\mapsto X_i\, \, \textup{ for all } i, 1 \leq i \leq 8
$$
is surjective onto~\eqref{eq: Cox ring Aa_MM description}. By additionally quotienting by $\mathcal{J}$, we obtain a surjection onto a ring isomorphic to the Cox ring. Thus, by computing the kernel of the composition of $\tilde{G}$ with this quotient map, we obtain a presentation of the Cox ring.  Another Macaulay2 computation reveals that this kernel is $\langle u_6-1, u_5-1, u_1 u_2 + u_3 u_4 - u_7 u_8 \rangle$. Getting rid of the extraneous variables $u_5, u_6$, and relabeling variables by
\[
 x:=u_7, \qquad y:=u_8,
\]
we obtain that the Cox ring may be described as follows.

\begin{Proposition}\label{prop: Cox ring for rank 2 ex}
Following the notation above, the Cox ring of $X_{\Aa_s}(\PP)$ is isomorphic to 
$$
R_\PP:= \C[u_1,u_2, u_3, u_4, x, y] / \langle xy-u_1 u_2 - u_3 u_4 \rangle.
$$
\end{Proposition}

\begin{Remark}\label{Rem:grading}
A remark is in order regarding the grading on the Cox ring as presented in Proposition~\ref{prop: Cox ring for rank 2 ex}. In ~\eqref{eq: Cox ring Aa_MM description} we described the ring $\Gamma(X,\mathcal{S})$ in the language of detropicalizations. This ring is naturally graded by $\Z^\ell$, by its definition, and in our case $\ell=4$. However, the Cox ring is graded not by $\Z^\ell$ but by the class group ${\mathrm{Cl}}(X)$, which is a quotient of $\Z^\ell$. Indeed, the ideal $\mathcal{J}$ is not homogeneous with respect to the $\Z^\ell$-grading. It is instead homogeneous with respect to a $\Z^\ell/L$-grading where $L$ is the subgroup generated by the vector $(1,1,-1,-1)$. This quotient is precisely ${\mathrm{Cl}}(X)$, and hence the quotient ring obtained by modding out~\eqref{eq: Cox ring Aa_MM description} by $\mathcal{J}$ is graded by ${\mathrm{Cl}}(X) \cong \Z^4/L$ as required.  
\end{Remark}

\section{Connection to complexity-one and polyhedral divisors}\label{sec:complexity-one}

The computation we gave in Section~\ref{sec: cox ring} of the Cox ring $R_\PP$ of $X_{\Aa_s}(\PP)$ can be rephrased in the language of \emph{(proper) polyhedral divisors}, following work of Ilten \cite{Ilten2011, Ilten2012}. In this section we make this connection concretely by phrasing the computation in Section~\ref{sec: cox ring} in this language. For brevity, we assume that the reader is familiar with the language of polyhedral divisors (for definitions and more details, see e.g. \cite{Altmann-Hausen}). Separately, the toric degenerations associated to $X_{\Aa_s}(\PP)$ can also be placed in the theory of complexity-$1$ varieties; Oda uses this connection and the divisorial fans appearing in the work of Ilten and Vollmert \cite{IltenVollmert} to concretely describe the global family in \cite[Section 7]{Oda2025}.

 In Remark~\ref{Rem:grading} we observed that the Cox ring of $X_{\Aa_s}(\PP)$ is equipped with a grading by ${\mathrm{Cl}}(X) \cong \Z^4/L$. Using the fact that $X_{\Aa_s}(\PP)$ is a complexity-$1$ surface, we may extend this to a grading by $\Z^4$. This leads to an action on the $x$ and $y$ variables by weights $1$ and $-1$ respectively. 
 
Let us concretely write out the matrix which encodes the grading. This takes several steps. 
First, we note from Section~\ref{sec: cox ring} that the generators $x,y$ in  Proposition~\ref{prop: Cox ring for rank 2 ex} correspond to the monomials $x_1 t_1t_2^{-1} t_3^{-1} t_4^2$ and $x_2 y_1^{-1} t_1^{-1}t_2 t_3^2 t_4^{-1}$ respectively. Computing the order of vanishing along the divisors $D_1, D_2, D_3, D_4$ by pairing the exponent vectors of these monomials with the associated points, we may compute that the orders of vanishing are as recorded in the matrix below; here the matrix has column labels $x,y, u_1,u_2, u_3, u_4$ and the rows are labelled by the divisors $D_1, D_2, D_3, D_4$. 
\[
\begin{bmatrix} 
1 & -1 & 1 & 0 & 0 & 0 \\
-1 & 1 & 0 & 1 & 0 & 0 \\
-1 & 2 & 0 & 0 & 1 & 0 \\
2 & -1 & 0 & 0 & 0 & 1 \\
\end{bmatrix}  
\]

Next, note that there is a lineality space spanned by $(1,1,-1,-1)^T$ (here $^T$ denotes transpose), as we saw already in Remark~\ref{Rem:grading}. Quotienting by the span of this vector corresponds to viewing the columns as elements of $\Z^4/L \cong {\mathrm{Cl}}(X)$, so the ring is graded by ${\mathrm{Cl}}(X)$ and not $\Z^4$. We may agree to choose the representative in $\Z^4/L$ which has a $0$ in the top coordinate. With this choice, we obtain the new matrix 
\[
\begin{bmatrix} 
0 & 0 & 0 & 0 & 0 & 0 \\
-2 & 2 & -1 & 1 & 0 & 0 \\
0 & 1 & 1 & 0 & 1 & 0 \\
3 & -2 & 1 & 0 & 0 & 1 \\
\end{bmatrix} 
\]
Note that we are no longer in the situation where the columns can be interpreted as orders of vanishing of the generators. 
Since the top row is just the zero vector, we omit it and obtain the $3 \times 6$ matrix 
\[
\begin{bmatrix} 
-2 & 2 & -1 & 1 & 0 & 0 \\
0 & 1 & 1 & 0 & 1 & 0 \\
3 & -2 & 1 & 0 & 0 & 1 \\
\end{bmatrix} 
\]
Incorporating the action on the $x$ and $y$ by weights $1$ and $-1$ and inserting this data in the bottom row, we obtain the matrix below. 
\[A = \begin{bmatrix}
-2 & 2 & -1 & 1 & 0 & 0\\
0 & 1 & 1 & 0 & 1 & 0\\
3 & -2 & 1 & 0 & 0 & 1\\
1 & -1 & 0 & 0 & 0 & 0
\end{bmatrix}\]

Altmann and Hausen \cite{Altmann-Hausen} establish that the coordinate ring of any normal affine variety $X$ with an effective action of a torus $T$ is the section ring of some (proper) polyhedral divisor $\mathfrak{D} \in \mathrm{Pol}(\mathcal Y, N, \sigma)$, where $\mathcal Y$ is a variety of dimension $\dim_\C(X) - \dim_\C(T)$, $N$ is the cocharacter lattice of $T$, and $\sigma$ is a cone in $N$.  
We now show how to write $R_\PP$ as the section ring of a polyhedral divisor over a surface using the recipe in \cite[Section 11]{Altmann-Hausen}.

Let $\mathcal{X}=V(xy - u_1u_2 - u_3u_4)\subseteq \C^6$ be the variety corresponding to $R_\PP$ equipped with the action of the torus $T$ with $\dim_\C(T)=4$ given by the matrix $A$. 
Let $N_T=\Z^4$ and $N_K=\Z^6$ be the lattices of one parameter subgroups of $T$ and $K =(\C^*)^6$, respectively.
Consider the split exact sequence 
\begin{equation*}
    \xymatrix{
        0 \ar[r] & N_T \ar[r]^{\phi} & N_K \ar[r]^{\rho} \ar@/^1pc/[l]^{\eta} & N_{\mathcal Y} \ar[r] & 0
    }
\end{equation*}
where $N_{\mathcal Y} :=N_K/N_T \cong\Z^2$ and the maps are given by the matrices
\begin{equation}\label{eq: phi rho eta}
    \phi=A^{T},
    \quad
    \rho=\begin{bmatrix} -1&-1&1&1&0&0\\-1&-1&0&0&1&1\end{bmatrix},
    \quad
    \eta=\begin{bmatrix}0&0&0&1&0&0\\0&0&0&0&1&0\\0&0&0&0&0&1\\1&0&0&2&0&-3\end{bmatrix}
.\end{equation}
Let $\{e_1,\cdots,e_6\}$ be the standard basis of $N_K = \Z^6$ and $\delta$ be the positive orthant $\delta = \mathrm{Cone}\{e_1,\cdots,e_6\}$ in $N_K\otimes \Q \cong \Q^6$. 
Define the variety $\mathcal Y$ to be the toric variety corresponding to the coarsest fan $\Sigma_{\mathcal Y}$ in $N_{\mathcal Y}\otimes\Q$ that refines all cones $\rho(\delta_0)$, where $\delta_0$ is a face of $\delta$.
From the definition of $\rho$ in~\eqref{eq: phi rho eta} we can see that $\Sigma_{\mathcal Y}$ is the fan of $\mathbb{P}^2$, so $\mathcal Y \cong \mathbb{P}^2$.
Define the cone $\sigma$ as $\sigma :=\eta(\delta\cap(\phi(N_T)\otimes\Q))\subseteq N_T\otimes \Q$. A computation yields that $\sigma$ is generated by the columns of the matrix
\begin{equation*}
    \hat\sigma=\begin{bmatrix}
        0&1&0&0&1&0&1&1\\
        1&0&0&0&0&1&1&1\\
        0&1&1&1&1&0&0&0\\
        0&0&-3&-2&-1&1&2&3
    \end{bmatrix}
.\end{equation*}
Next, consider the primitive ray generators $\mathfrak{r}_0=(1,0)$, $\mathfrak{r}_1=(0,1)$, and $\mathfrak{r}_2=(-1,-1)$ of the $1$-dimensional rays of the fan $\Sigma_{\mathcal{Y}}$. For each $i \in \{0,1,2\}$, define 
\begin{equation*} \Delta_i := \eta(\delta\cap\rho^{-1}(\mathfrak{r}_i))\subset N_T\otimes\Q.
\end{equation*}
This is a polyhedron in $N_T\otimes\Q$. A computation shows that 
\begin{align*}
    \Delta_0&=\sigma+\mathrm{Conv}((0,0,0,0),(1,0,0,2)),
    \\
    \Delta_1&=\sigma+\mathrm{Conv}((0,1,0,0),(0,0,1,-3)),
    \\
    \Delta_2&=\sigma+\mathrm{Conv}((0,0,0,0),(0,0,0,1)).
\end{align*}
Note that the cone $\sigma$ is the recession cone of each $\Delta_i$. 
Let $E_i$ be the toric divisor in $\mathcal{Y}$ corresponding to the $1$-dimensional ray in $\Sigma_{\mathcal{Y}}$ corresponding to $\mathfrak{r}_i$. Then we may define the polyhedral divisor (where the coefficients are convex polyhedra in $N_T\otimes\Q$, see \cite{Altmann-Hausen}) $\mathfrak{D}_{\mathrm{toric}}$ 
by 
\begin{equation*}
    \mathfrak{D}_{\mathrm{toric}}= \Delta_0 \otimes E_0 + \Delta_1 \otimes E_1 + \Delta_2 \otimes E_2 \in \mathrm{Pol}(\mathbb{P}^2, \Z^4, \sigma).
\end{equation*}

Next, consider the following map.
\begin{equation*}
    (\C^*)^6\to (\C^*)^2\subseteq\mathbb{P}^2,\qquad (t_1,\ldots,t_6)\mapsto\left(\frac{t_3t_4}{t_1t_2},\frac{t_5t_6}{t_1t_2}\right)
.\end{equation*}
The closure of the image of $\mathcal{X}\cap (\C^*)^6$ in $\mathbb{P}^2$ is the curve
\begin{equation*}
    \mathcal{Z}=V(w_0-w_1-w_2),
\end{equation*}
where $(w_0:w_1:w_2)$ are the homogeneous coordinates of $\mathbb{P}^2$.
We then have that 
\begin{equation*}
    R_\PP \cong A(\mathcal{Z}, \mathfrak{D}_{\mathrm{toric}}\mid_{\mathcal{Z}})
,\end{equation*}
where $A(\mathcal{Z}, \mathfrak{D}_{\mathrm{toric}}\mid_{\mathcal{Z}})$ is the section ring
\[A(\mathcal Z, \mathfrak{D}_{\mathrm{toric}}\mid_{\mathcal{Z}}) = \bigoplus_{u \in \sigma^\vee \cap N^\vee} H^0(\mathcal Z, \mathfrak{D}_{\mathrm{toric}}\mid_{\mathcal{Z}}(u)).\]

\bibliographystyle{amsplain}
\bibliography{biblio}

\end{document}